\documentclass[11pt,reqno]{amsart}
\usepackage{amsmath, amssymb, amsfonts, xcolor}
\usepackage[hidelinks]{hyperref}
\numberwithin{equation}{section}
\usepackage{amsthm}
\newtheorem{thm}{Theorem}[section]

\newtheorem{prop}[thm]{Proposition}
\newtheorem{lem}[thm]{Lemma}
\newtheorem{conj}[thm]{Conjecture}
\newtheorem{dfn}[thm]{Definition}

\newtheorem{remark}[thm]{Remark}
\newtheorem{cor}[thm]{Corollary}
\newtheorem{exple}[thm]{Example}

\usepackage[a4paper, top = 1.2in, bottom = 1.2in, left = 1.2in, right = 1.2in]{geometry}

\usepackage{enumitem}
\newlist{steps}{enumerate}{1}
\setlist[steps, 1]{label = Step \arabic*:}

\numberwithin{equation}{section}

\newcommand{\A}{\mathbb{A}}
\newcommand{\C}{\mathbb{C}}
\newcommand{\F}{\mathbb{F}}
\newcommand{\N}{\mathbb{N}}
\newcommand{\Q}{\mathbb{Q}}

\newcommand{\R}{\mathbb{R}}

\newcommand{\Z}{\mathbb{Z}}

\newcommand{\mbH}{\mathbb{H}}

\newcommand{\mcO}{\mathcal{O}}

\newcommand{\mfc}{\mathfrak{c}}

\newcommand{\mfn}{\mathfrak{n}}

\newcommand{\mfp}{\mathfrak{p}}
\newcommand{\mfq}{\mathfrak{q}}
\newcommand{\mfP}{\mathfrak{P}}

\newcommand{\SL}{\mathrm{SL}}

\newcommand{\GL}{\mathrm{GL}}

\newcommand{\Gal}{\mathrm{Gal}}
\newcommand{\Frob}{\mathrm{Frob}}

\newcommand{\End}{\mathrm{End}}
\newcommand{\Aut}{\mathrm{Aut}}
\newcommand{\Tr}{\mathrm{Tr}}

\def\1{1\!\!1}

\newcommand{\bsmat}[4]{\bigl[ \begin{smallmatrix} #1 & #2 \\ #3 & #4 \end{smallmatrix} \bigr]}

\title[Generalized Fermat equation over $K$]{Generalized Fermat equation over number fields}

\author[S. Sahoo]{Satyabrat Sahoo}

\address[S. Sahoo]{Shanghai Institute for Mathematics and Interdisciplinary Sciences (SIMIS), Shanghai 200433, China
	\& \newline
	Research Institute of Intelligent Complex Systems, Fudan University, Shanghai 200433, China}
\email{sahoos@simis.cn}

\date{October 29, 2026}

\begin{document}
\begin{abstract}
Let $K$ be a number field and $A,B,C\in\mcO_K\setminus\{0\}$.
Put $S_K=\{\mfP:\mfP\mid2\}$ and
$S_K'=\{\mfq:\mfq\mid2ABC\}$. We establish two criteria for the
non-existence of non-trivial solutions of $Ax^p+By^p=Cz^p$ over $K$
for sufficiently large prime exponents $p$. Both criteria use a valuation bound
on the $S_K'$-unit equation at a fixed $\mfP\in S_K$. The first excludes all
non-trivial solutions in $K^3$ when $v_{\mfP}(A)=v_{\mfP}(B)=m$ and
$v_{\mfP}(C)=m+k$, with $k>4v_{\mfP}(2)$. The second, under an
explicit nonvanishing condition on the coefficient valuations,
excludes solutions $(a,b,c)\in(K^\times)^3$ whose $\mfP$-valuations
are not all equal. In particular, it excludes non-trivial primitive
integral solutions with $\mfP\mid abc$. Neither criterion restricts the residue degree of $\mfP$. Both criteria hold unconditionally over every totally real field;
over general number fields, we assume residual modularity and
Eichler--Shimura-type conjectures. The first criterion covers a new
coefficient range. In particular, under the corresponding $S_K$-unit bound, $x^p+y^p=2^rz^p$ has no non-trivial
solutions in $K^3$ for each fixed $r\geq5$ and sufficiently large
prime $p$. This bound holds for a set of real quadratic fields of
relative density $1$. For a family of imaginary quadratic fields of relative density $5/6$,
we explicitly determine the $S_K$-unit solutions and verify the
valuation bound for both criteria when $A,B,C\in\{\pm2^r:r\in\Z_{\geq0}\}$, including
the same $r\geq5$ conclusion.
\end{abstract}

	\maketitle
	\section{Introduction}

	Diophantine equations constitute one of the oldest and most active areas
	of research in number theory. The most celebrated example is
	Fermat's equation
	$
	x^n+y^n=z^n,$ with $ n\geq 3.
	$
	Fermat's Last Theorem asserts that this equation has no non-trivial integer
	solutions. Wiles proved this theorem, completing the
	argument jointly with Taylor, through the modularity of semistable elliptic curves
	over $\Q$; see \cite{W95,TW95}. A central ingredient in the proof is the
	modular method, which combines the arithmetic of Frey elliptic curves
	with results on mod-$p$ Galois representations, including irreducibility
	results (cf.~\cite{M78}) and Ribet's level-lowering theorem
	(cf.~\cite{R90}).

	The success of the modular method in the proof of Fermat's Last Theorem
	has motivated extensive study of the more general class of Diophantine
	equations
	\begin{equation}
		\label{generalized Fermat eqn}
		Ax^p+By^q=Cz^r,
		\qquad
		A,B,C\in\Z\setminus\{0\},
	\end{equation}
	where $A,B,C$ are fixed coprime integers and $p,q,r\geq2$ are integers.
	Such equations are commonly referred to as \emph{generalized Fermat
		equations}, and the triple $(p,q,r)$ is called the \emph{signature} of
	\eqref{generalized Fermat eqn}.
	The modular method has been
	particularly successful in this direction, notably for equations of
	signatures $(p,p,p)$, $(p,p,2)$, and $(p,p,3)$.

    For signature $(p,p,p)$, Freitas and Siksek \cite{FS15}
    established $S$-unit criteria for asymptotic non-existence of
    solutions to $x^p+y^p+z^p=0$ over totally real number fields.
    Their approach relates solutions for sufficiently large prime
    exponents to solutions of certain $S$-unit equations; see
    Definition~\ref{asymptotic solution} for the asymptotic terminology.

	Deconinck \cite{D16} subsequently extended this framework to the
	generalized Fermat equation
	$
	Ax^p+By^p=Cz^p
	$
	over totally real number fields, where $A,B,C\in\mcO_K$ and $ABC$ is
	odd, in the sense that
	$
	\mfP\nmid ABC$
	$\text{for every }\mfP\in S_K.
	$
	The restriction that the ground field be totally real is closely
	connected with the modularity results available for elliptic curves
	over totally real fields. Under Conjectures~\ref{SS conj1} and~\ref{conj 2}, \c{S}eng\"un--Siksek \cite{SS18}
	established an $S$-unit criterion for the asymptotic Fermat equation
	$
	x^p+y^p+z^p=0
	$
	over an arbitrary number field. Their criterion is formulated using
	primes above $2$ having residue degree one. Kara--Ozman \cite{KO20}
	subsequently extended this framework to the generalized equation
	$
	Ax^p+By^p=Cz^p
	$
	over arbitrary number fields, under the additional assumption that
	$A,B,C$ are odd at every prime above $2$.

When a coefficient is divisible by a prime above $2$, the local
behaviour changes. Over $\Q$, Ribet~\cite{R97} proved that
$x^p+y^p=2^rz^p$ has no non-trivial coprime integer solutions
for $2\leq r<p$. Over totally real number fields $K$, Kumar and
Sahoo~\cite{KS24 Diop1} studied the same equation and established
asymptotic non-existence in $\mcO_K^3$ for $r=2,3$ under an
$S$-unit bound. Sahoo~\cite{S26} subsequently treated
$Ax^p+By^p=Cz^p$ with $A,B,C\in\mcO_K\setminus\{0\}$,
obtaining asymptotic non-existence in $K^3$ under an $S$-unit
bound and additional hypotheses, including the coefficient bound
$\max\{v_{\mfP}(A),v_{\mfP}(BC)\}\leq4v_{\mfP}(2)$
at a fixed $\mfP\in S_K$.

We study non-trivial solutions $(x,y,z)\in K^3$ of
\begin{equation}
    \label{GFE introduction}
    Ax^p+By^p=Cz^p,
    \qquad A,B,C\in\mcO_K\setminus\{0\},
\end{equation}
for sufficiently large prime exponents $p$, using local conditions
on the coefficients and a valuation bound for solutions of an
associated $S$-unit equation.

Theorem~\ref{main result1 for Ax^p+By^p=Cz^p} gives
\emph{two criteria over arbitrary number fields}, assuming
Conjecture~\ref{conj 1} on residual modularity at the Serre
conductor (a special case of Conjecture~\ref{SS conj1}), and
the Eichler--Shimura-type Conjecture~\ref{conj 2}. Both require
the same valuation bound at a fixed $\mfP\in S_K$ for every
solution of the $S_K'$-unit equation, where
$S_K'=\{\mfq:\mfq\mid2ABC\}$.
Part~(1) excludes all non-trivial solutions in $K^3$ for
sufficiently large prime exponents $p$ when
$v_{\mfP}(A)=v_{\mfP}(B)=m$ and $v_{\mfP}(C)=m+k$, with
$k>4v_{\mfP}(2)$. Part~(2) gives the corresponding asymptotic
non-existence result for solutions $(a,b,c)\in(K^\times)^3$ whose valuations
$v_{\mfP}(a),v_{\mfP}(b),v_{\mfP}(c)$ are not all equal, provided
$\alpha_{\mfP}\beta_{\mfP}\gamma_{\mfP}\neq0$, where these quantities
are defined below. In particular, part~(2) excludes non-trivial normalized integral
solutions $(a,b,c)\in\mcO_K^3$ with $\mfP\mid abc$
for all sufficiently large prime exponents $p$. Both criteria allow even coefficients and arbitrary
residue degree at $\mfP$, extending the coefficient and local
settings treated in \cite{SS18,KO20}.

Theorem~\ref{totally real unrestricted criterion} proves
\emph{both parts of Theorem~\ref{main result1 for Ax^p+By^p=Cz^p}
unconditionally over every totally real field}, under the respective
coefficient conditions and the same $S_K'$-unit bound.
Part~(1) extends the coefficient range treated in
\cite{KS24 Diop1,S26}. For $x^p+y^p=2^rz^p$, \cite[Theorem~3.4]{KS24 Diop1} treats
$r=2,3$, while the coefficient bound in
\cite[Theorem~2.7]{S26} forces $r\leq4$.
Corollary~\ref{totally real powers of two} covers
\emph{$x^p+y^p=2^rz^p$ for every fixed $r\geq5$}, excluding all non-trivial solutions
in $K^3$ under the stated $S_K$-unit bound.

We verify the required bound for solutions of the $S_K$-unit equation
for a family of real quadratic fields of relative density $1$
(Theorem~\ref{real quadratic generalized Fermat}) and a family of
imaginary quadratic fields of relative density $5/6$
(Theorem~\ref{imaginary quadratic generalized Fermat} and
Proposition~\ref{thm for density}). The resulting asymptotic
non-existence statements over imaginary quadratic fields are
conditional on the stated conjectures.

The modular method over number fields has also been developed
for other signatures. Generalized
	Fermat equations of signature $(p,p,2)$,
	$
	Ax^p+By^p+Cz^2=0,
	$
	have been studied, for example, in
	\cite{IKO20,M22,KS24 Diop1,KS Diop2,CIKO26}, while equations of
	signature $(p,p,3)$,
	$
	Ax^p+By^p+Cz^3=0,
	$
	have been investigated in
	\cite{M22,IKO23,KS Diop3,KNO26}.

\subsection{Main results}
\label{notations section for Ax^p+By^p=Cz^p}
Let $K$ be a number field and $A,B,C \in \mcO_K\setminus \{0\}$. In this section, we state the main results for the solutions of the generalized Fermat equation
\begin{equation}
	\label{Ax^p+By^p=Cz^p}
	Ax^p+By^p=Cz^p
\end{equation}
over the number field $K$ with prime exponent $p\geq 3$.
For each non-zero prime ideal $\mfq$, we normalize the additive
valuation by $v_{\mfq}(K^\times)=\Z$.
Let $P$ denote the set of all non-zero prime ideals of $\mcO_K$, and put
$
S_K:=\{\mfP\in P:\mfP\mid2\}$ and $
S_K^{\prime}:=\{\mfq\in P:\mfq\mid2ABC\}.
$

\begin{dfn}[Trivial solution]
	A solution $(a, b, c)\in K^3$ to equation~\eqref{Ax^p+By^p=Cz^p} is said to be trivial if $abc=0$ and non-trivial otherwise.
	We say $(a, b, c)\in \mcO_K^3$ is primitive if $a\mcO_K+b\mcO_K+c\mcO_K=\mcO_K$.
\end{dfn}
\begin{dfn}
    \label{def for W_K}
    For $\mfP\in S_K$, let $W_{K,\mfP}$ be the set of all triples
    $(a,b,c)\in(K^\times)^3$ such that
    $
        \max\{v_{\mfP}(a),v_{\mfP}(b),v_{\mfP}(c)\}
        >\min\{v_{\mfP}(a),v_{\mfP}(b),v_{\mfP}(c)\}.$
    Equivalently, the valuations $v_{\mfP}(a)$, $v_{\mfP}(b)$,
    and $v_{\mfP}(c)$ are not all equal. This condition
    is unchanged by multiplying all coordinates by the same
    element of $K^\times$. For a non-trivial primitive integral
    solution $(a,b,c)\in\mcO_K^3$, it is equivalent to
    $\mfP\mid abc$.
\end{dfn}

\begin{dfn}
    \label{asymptotic solution}
    We say that $Ax^p+By^p=Cz^p$ has no asymptotic solution
    in a subset $S\subseteq K^3$ if there exists a constant
    $V>0$ (depending on $K,A,B,C$) such that,
    for every prime $p>V$, the equation has no non-trivial
    solution in $S$.
\end{dfn}

\begin{remark}
    \label{remark for W_K}
    Let $(a,b,c)\in W_{K,\mfP}$ satisfy \eqref{Ax^p+By^p=Cz^p},
    with $p>\max\{v_{\mfP}(A),v_{\mfP}(B),v_{\mfP}(C)\}$.
    Set $t=\min\{v_{\mfP}(a),v_{\mfP}(b),v_{\mfP}(c)\}$.
    Replacing $(a,b,c)$ by $(\xi a,\xi b,\xi c)$, where
    $\xi\in K^\times$ satisfies $v_{\mfP}(\xi)=-t$, we may assume that
    $
        \min\{v_{\mfP}(a),v_{\mfP}(b),v_{\mfP}(c)\}=0.
    $
    If $v_{\mfP}(a),v_{\mfP}(b)>0$, then
    \[
    \begin{aligned}
        v_{\mfP}(C)+p\,v_{\mfP}(c)
        &=v_{\mfP}(Aa^p+Bb^p)\\
        &\geq\min\{v_{\mfP}(Aa^p),v_{\mfP}(Bb^p)\}
        \geq p.
    \end{aligned}
    \]
    Since $v_{\mfP}(C)<p$, this gives $v_{\mfP}(c)>0$,
    a contradiction. The other cases are analogous. Thus
    $
        \bigl(v_{\mfP}(a),v_{\mfP}(b),v_{\mfP}(c)\bigr)
        \in\{(s,0,0),(0,s,0),(0,0,s):s\in\Z_{>0}\}.
    $
\end{remark}

For any set $S \subseteq P$, let $\mcO_{S}:=\{\alpha \in K : v_\mfq(\alpha)\geq 0 \text{ for all } \mfq \in P \setminus S\}$ be the ring of $S$-integers in $K$ and let $\mcO_{S}^*$ denote its unit group. For $\mfP\in S_K$ and $A,B,C\in\mcO_K\setminus\{0\}$, put
\begin{align*}
	\alpha_{\mfP}
	&:=8v_{\mfP}(2)+v_{\mfP}(BCA^{-2}),\\
	\beta_{\mfP}
	&:=8v_{\mfP}(2)+v_{\mfP}(ACB^{-2}),\\
	\gamma_{\mfP}
	&:=8v_{\mfP}(2)+v_{\mfP}(ABC^{-2}),
\end{align*}
\noindent\textbf{Normalization.}
Fix a set $\mathcal H=\{\mathfrak m_1,\ldots,\mathfrak m_h\}$
of prime ideals outside $S_K'$, containing exactly one representative
of each class in $\operatorname{Cl}(K)$, where
$h=\#\operatorname{Cl}(K)$. Such a choice exists since every ideal
class contains infinitely many prime ideals. We fix $\mathcal H$
independently of $p$ and the solution. An integral solution
$(a,b,c)$ is called \emph{normalized} if
\[
    a\mcO_K+b\mcO_K+c\mcO_K=\mathfrak m
    \qquad\text{for some }\mathfrak m\in\mathcal H.
\]
Consequently,
\[
    \min\{v_{\mfq}(a),v_{\mfq}(b),v_{\mfq}(c)\}
    =v_{\mfq}(\mathfrak m)
    =\begin{cases}1,&\mfq=\mathfrak m,\\0,&\mfq\neq\mathfrak m.\end{cases}
\]
In particular, the minimum is zero at every $\mfq \in S_K'$.
Every non-trivial solution $(a,b,c)\in K^3$ admits a normalized
integral representative $(\xi a,\xi b,\xi c)$ for some
$\xi\in K^\times$; see \S\ref{steps to prove main results} for more details.

We now state two criteria for asymptotic non-existence.
The first applies to all non-trivial solutions in $K^3$;
the second concerns solutions in $W_{K,\mfP}$.
\begin{thm}
    \label{main result1 for Ax^p+By^p=Cz^p}
    \label{main result without divisibility}
    Let $K$ be a number field for which Conjectures~\ref{conj 1}
    and~\ref{conj 2} hold, and let
    $A,B,C\in\mcO_K\setminus\{0\}$.
    Fix a prime $\mfP\in S_K$ and suppose that every solution
    $(\lambda,\mu)$ of the $S_K'$-unit equation
    \begin{equation}
        \label{S_K-unit solution}
        \lambda+\mu=1,\qquad
        \lambda,\mu\in\mcO_{S_K'}^\times,
    \end{equation}
    satisfies
    \begin{equation}
        \label{assumption for main result x^p+y^p=2^rz^p}
        \max\{|v_{\mfP}(\lambda)|,|v_{\mfP}(\mu)|\}
            \leq4v_{\mfP}(2).
    \end{equation}
    Then the following statements hold.
    \begin{enumerate}
    	\item If $
    		v_{\mfP}(A)=v_{\mfP}(B)=m,\
    		v_{\mfP}(C)=m+k,
    	$
    	for some integers $m\geq0$ and $k>4v_{\mfP}(2)$, then
    	$Ax^p+By^p=Cz^p$ has no asymptotic solution in
    	$K^3$.
        \item If $\alpha_{\mfP}\beta_{\mfP}\gamma_{\mfP}\neq0$,
        then $Ax^p+By^p=Cz^p$ has no asymptotic solution in
        $W_{K,\mfP}$. In particular, there are no non-trivial normalized integral
        solutions $(a,b,c)\in\mcO_K^3$ with $\mfP\mid abc$
        for all sufficiently large prime exponents $p$.
    \end{enumerate}
\end{thm}
Over totally real fields, both parts hold without assuming
Conjectures~\ref{conj 1} and~\ref{conj 2}.
\begin{thm}
    \label{totally real unrestricted criterion}
    Let $K$ be a totally real number field and let
    $A,B,C\in\mcO_K\setminus\{0\}$. Put
    $S_K'=\{\mfq:\mfq\mid2ABC\}$.
    Fix $\mfP\in S_K$ and suppose that every solution
    $(\lambda,\mu)$ of
    \[
        \lambda+\mu=1,\qquad
        \lambda,\mu\in\mcO_{S_K'}^\times,
    \]
    satisfies
    $
        \max\{|v_{\mfP}(\lambda)|,|v_{\mfP}(\mu)|\}
            \leq4v_{\mfP}(2).
    $
    Then the following statements hold.
    \begin{enumerate}
        \item If $v_{\mfP}(A)=v_{\mfP}(B)=m$ and
        $v_{\mfP}(C)=m+k$ for integers $m\geq0$ and
        $k>4v_{\mfP}(2)$, then $Ax^p+By^p=Cz^p$ has no
        asymptotic solution in $K^3$.
        \item If $\alpha_{\mfP}\beta_{\mfP}\gamma_{\mfP}\neq0$,
        then $Ax^p+By^p=Cz^p$ has no asymptotic solution in
        $W_{K,\mfP}$.
    \end{enumerate}
\end{thm}

\begin{remark}
    Theorem~\ref{totally real unrestricted criterion}(1) improves on
    \cite[Theorem~2.7]{S26} by covering a new range of coefficients.
    At the prime $\mfP$ in part~(1),
    we have $v_{\mfP}(BC)=2m+k>4v_{\mfP}(2)$, whereas
    \cite[Theorem~2.7]{S26} requires
    $\max\{v_{\mfP}(A),v_{\mfP}(BC)\}\leq4v_{\mfP}(2)$
    at the prime used in its $S_K'$-unit criterion.
    In particular, $A=B=1$ and $C=2^r$ with $r\geq5$ violate
    the latter bound at every prime above $2$, but are covered
    by Corollary~\ref{totally real powers of two} below.
    Moreover, Theorem~\ref{totally real unrestricted criterion}(1)
    holds over every totally real field without the $(ES)$
    hypothesis (odd degree or the Eichler--Shimura conjecture)
    or the congruence conditions modulo $3$ in
    \cite[Theorem~2.7]{S26}.
\end{remark}

\begin{remark}
	Part~(2) extends \cite[Theorem~2.5]{S26} from primitive
	integral solutions to all solutions in $W_{K,\mfP}\subset K^3$.
	For $\xi\in K^\times$, each valuation of $(\xi a,\xi b,\xi c)$
	is obtained by adding $v_{\mfP}(\xi)$ to the corresponding
	valuation of $(a,b,c)$. Thus membership in $W_{K,\mfP}$ is
	preserved. For a normalized integral solution,
	$\min\{v_{\mfP}(a),v_{\mfP}(b),v_{\mfP}(c)\}=0$; hence
	$(a,b,c)\in W_{K,\mfP}$ if and only if $\mfP\mid abc$.
\end{remark}
\begin{cor}
    \label{totally real powers of two}
    Let $K$ be a totally real number field and put
    $S_K=\{\mfP:\mfP\mid2\}$.
    Suppose that there is a prime $\mfP\in S_K$ such that every
    solution $(\lambda,\mu)$ of
   $
        \lambda+\mu=1,\
        \lambda,\mu\in\mcO_{S_K}^\times,
    $
    satisfies
    $
        \max\{|v_{\mfP}(\lambda)|,|v_{\mfP}(\mu)|\}
            \leq4v_{\mfP}(2).
    $
    Then, for every fixed integer $r\geq5$, the equation
    $
        x^p+y^p=2^rz^p
    $
    has no asymptotic solution in $K^3$.
\end{cor}
\begin{proof}
    Apply Theorem~\ref{totally real unrestricted criterion}(1)
    with $A=B=1$ and $C=2^r$. Then $S_K'=S_K$, and the
    coefficient condition holds with $m=0$ and
    $k=rv_{\mfP}(2)>4v_{\mfP}(2)$.
\end{proof}

\begin{remark}
	The proof of Theorem~\ref{main result1 for Ax^p+By^p=Cz^p} relies on the explicit bound \eqref{assumption for main result x^p+y^p=2^rz^p} for all the solutions of the $S_K^\prime$-unit equation~\eqref{S_K-unit solution}. In \cite{S14}, Siegel proved that for any finite set $S \subseteq P$, the $S$-unit equation has only a finite number of solutions over any number field $K$, and hence the $S_K^\prime$-unit equation~\eqref{S_K-unit solution} has only a finite number of solutions over $K$.
\end{remark}

Let $h_K^+$ denote the narrow class number of $K$.
We now consider coefficients $A,B,C \in \{u 2^rd^s: u\in\mcO_K^\times, r,s \in \Z_{\geq 0}\}$, where $d$ is an odd prime.
	\begin{prop}
	\label{prop for solution of S-unit eqn even soln}
	\label{S unit crit 2d}
	Let $K$ be a number field for which Conjectures~\ref{conj 1} and~\ref{conj 2} hold. Let $d \geq 3$ be a prime with $d \equiv 1 \pmod 4$. Assume
	\begin{enumerate}
		\item $2 \nmid h_K^+$;
		\item $2$ is inert in $K$, with $2\mcO_K =\mfP$;
		\item $d$ is inert in $K$;
		\item $d^{[K: \Q]} \not\equiv 1,9,17,25\pmod {32}$.
	\end{enumerate}
    Then, for
    $A,B,C\in\{u2^rd^s:u\in\mcO_K^\times,\ r,s\in\Z_{\geq0}\}$,
    the conclusions of Theorem~\ref{main result1 for Ax^p+By^p=Cz^p}
    hold over $K$ under their respective coefficient hypotheses.
\end{prop}

We call $(2,-1)$, $(-1,2)$, and $(\frac12,\frac12)$ the irrelevant solutions of the $S_K^\prime$-unit equation~\eqref{S_K-unit solution}; all other solutions are called relevant.
The following is an immediate corollary of
Theorem~\ref{main result1 for Ax^p+By^p=Cz^p}.
\begin{cor}
    \label{irrelevant cor}
    Let $K,A,B,C$ be as in Theorem~\ref{main result1 for Ax^p+By^p=Cz^p} (resp. Theorem~\ref{totally real unrestricted criterion}).
    If the $S_K'$-unit equation~\eqref{S_K-unit solution}
    has only irrelevant solutions, then the conclusions of
    Theorem~\ref{main result1 for Ax^p+By^p=Cz^p} (resp. Theorem~\ref{totally real unrestricted criterion}) hold over $K$
    under their respective coefficient hypotheses.
\end{cor}
\subsection{Imaginary quadratic field case}
We apply both parts of
Theorem~\ref{main result1 for Ax^p+By^p=Cz^p} to imaginary
quadratic fields $K=\Q(\sqrt{-d})$ and then determine the
relative density of the admissible squarefree positive integers $d$.
\begin{thm}
	\label{imaginary quadratic generalized Fermat}
	\label{imaginary quadratic without divisibility}
    Let $d\geq5$ be squarefree with $-d\not\equiv1\pmod8$,
    and put $K=\Q(\sqrt{-d})$.
    Assume Conjectures~\ref{conj 1} and~\ref{conj 2}.
    Write $S_K=\{\mfP\}$.
    \begin{enumerate}
    	 \item For each fixed integer $r\geq5$, the equation
    	\begin{equation}
    		\label{unrestricted power of two equation}
    		x^p+y^p=2^r z^p
    	\end{equation}
    	has no asymptotic solution in $K^3$.
    	\item Let
    	$A,B,C\in\{\pm2^r:r\in\Z_{\geq0}\}$ satisfy
    	$\alpha_{\mfP}\beta_{\mfP}\gamma_{\mfP}\neq0$.
    	Then $Ax^p+By^p=Cz^p$ has no asymptotic solution in
    	$W_{K,\mfP}$.
    \end{enumerate}
\end{thm}

\begin{exple}
    For $K=\Q(\sqrt{-11})$, the prime $2$ is inert; put $\mfP=(2)$.
    Assuming Conjectures~\ref{conj 1} and~\ref{conj 2},
    Theorem~\ref{imaginary quadratic generalized Fermat}(1)
    applies to $x^p+y^p=32z^p$.
    Taking $A=B=1$ and $C=2$ gives
    $\alpha_{\mfP}=\beta_{\mfP}=9$ and $\gamma_{\mfP}=6$.
    Hence Theorem~\ref{imaginary quadratic generalized Fermat}(2)
    shows that $x^p+y^p=2z^p$ has no asymptotic solution in
    $W_{K,\mfP}$.
\end{exple}

\begin{remark}
    Theorem~\ref{imaginary quadratic generalized Fermat}(1) excludes
    all non-trivial solutions in $K^3$ for sufficiently large $p$,
    including those whose coordinate valuations at $\mfP$ are all
    equal. After scaling such a solution so that these valuations
    are zero, the associated Frey curve satisfies
    $
        v_{\mfP}(j_E)=(8-2r)v_{\mfP}(2).$
    This is negative exactly when $r\geq5$, explaining the range
    in part~(1). This calculation gives no non-existence conclusion
    for $r=2,3,4$.
    When $r=1$, the triple $(1,1,1)$ is a non-trivial primitive
    solution over $K$ for every prime $p$. It lies outside
    $W_{K,\mfP}$, so it does not contradict part~(2).
\end{remark}
To quantify this family, put
\[
\N^{\mathrm{sf}}:=\{d\in\Z_{\geq2}:d\text{ is squarefree}\}.
\]
\begin{dfn}
	\label{relative density}
	For $S \subseteq \N^{\text{sf}}$, the relative density of $S$ is defined by
	\[
	\delta_{\text{rel}}(S):=
	\lim_{x\to\infty}
	\frac{\#\{d\in S:d\leq x\}}
	{\#\{d\in\N^{\mathrm{sf}}:d\leq x\}},
	\]
	if the limit exists.
\end{dfn}
Next, we compute the density for the above result.
\begin{prop}
	\label{thm for density}
	Let $U:=\{d \in \N^{\text{sf}} :  d\geq5$ and
	$-d\not\equiv1\pmod8\}$. Then $ \delta_{\text{rel}}(U)=  \frac{5}{6}$.
\end{prop}
Proposition~\ref{thm for density} shows that the family of imaginary
quadratic fields covered by
Theorem~\ref{imaginary quadratic generalized Fermat}
has relative density $5/6$ among squarefree positive parameters.
The proofs of the theorem and the density calculation are given in \S\ref{section for loc criteria}.

\subsection{Real quadratic field case}
\label{real quadratic applications}
For $d\in\N^{\mathrm{sf}}$, put $K_d=\Q(\sqrt d)$ and
$S_d:=S_{K_d}=\{\mfP:\mfP\mid2\}$. Let $\mathcal G$ be the set of such $d$
for which every solution of
$
    \lambda+\mu=1,\  \lambda,\mu\in\mcO_{S_d}^{\times},
$
satisfies
\begin{equation}
    \label{real quadratic unit bound}
    \max\{|v_{\mfP}(\lambda)|,|v_{\mfP}(\mu)|\}
        \leq4v_{\mfP}(2)
    \qquad\text{for every }\mfP\in S_d.
\end{equation}

\begin{thm}
    \label{real quadratic generalized Fermat}
    The set $\mathcal G$ has relative density $1$.
    Let $d\in\mathcal G$ and put $K=\Q(\sqrt d)$. Suppose that
    $A,B,C\in\mcO_K\setminus\{0\}$ have no prime divisors away
    from $2$ and that, for some $\mfP\in S_K$, we have
    $
        v_{\mfP}(A)=v_{\mfP}(B)=m,\
        v_{\mfP}(C)=m+k,
        \ m\geq0,\quad k>4v_{\mfP}(2).
    $
    Then $Ax^p+By^p=Cz^p$ has no asymptotic solution in
    $K^3$.
    In particular, for each fixed integer $r\geq5$, the equation
    \[
        x^p+y^p=2^rz^p
    \]
    has no asymptotic solution in $K^3$.
\end{thm}
\begin{remark}
    Freitas and Siksek~\cite[Theorem~4]{FS15} proved unconditionally
    that $x^p+y^p=z^p$ has no non-trivial solutions over
    $\Q(\sqrt d)$ for all sufficiently large primes $p$ when $d$
    ranges over a set of positive squarefree integers of relative
    density $5/6$. For each fixed $r\geq5$,
    Theorem~\ref{real quadratic generalized Fermat} establishes
    the corresponding asymptotic non-existence result for
    $x^p+y^p=2^rz^p$ over $\Q(\sqrt d)$ on a set of $d$ of
    relative density $1$. Thus our unconditional theorem covers
    a family of greater relative density for these generalized
    Fermat equations. Freitas and Siksek also obtain density $1$
    for the ordinary Fermat equation under the Eichler--Shimura
    conjecture; Theorem~\ref{real quadratic generalized Fermat}
    requires no such assumption.
\end{remark}
The proof of Theorem~\ref{real quadratic generalized Fermat} is given in \S\ref{section real quadratic proofs}.

\subsection{Methodology}
\label{methodology}
For Theorems~\ref{main result1 for Ax^p+By^p=Cz^p}
and~\ref{totally real unrestricted criterion}, we normalize a
non-trivial solution in $K^3$ so that its coordinates are integral
and generate a prime ideal $\mathfrak m\notin S_K'$ from a fixed
finite set of ideal-class representatives. We attach the Frey curve
\[
    E_{a,b,c}:\quad y^2=x(x-Aa^p)(x+Bb^p).
\]
The local calculations give two criteria ensuring
$v_{\mfP}(j_E)<0$ and $p\mid\#\bar\rho_{E,p}(I_{\mfP})$
for sufficiently large $p$. In
Theorem~\ref{main result1 for Ax^p+By^p=Cz^p}(2), these follow
from $\mfP\mid abc$ for the normalized coordinates and
$\alpha_{\mfP}\beta_{\mfP}\gamma_{\mfP}\neq0$.
For Theorem~\ref{main result1 for Ax^p+By^p=Cz^p}(1), writing $e=v_{\mfP}(2)$ and $h=v_{\mfP}(c)$ gives
\[
    v_{\mfP}(j_E)=8e-2k-2ph<0.
\]
Since $k>4e$, the fixed term $8e-2k$ is non-zero, so
$p\nmid v_{\mfP}(j_E)$ for sufficiently large $p$, even when
$h=0$. This coefficient condition therefore treats solutions
whose normalized coordinates are all units at $\mfP$, extending
the admissible range of $A,B,C$. Neither criterion requires
$\mfP$ to have residue degree one.

For the general-field result, Theorem~\ref{main result1 for Ax^p+By^p=Cz^p},
we require the elliptic-curve modularity
Conjecture~\ref{conj 1}, a special case of
Conjecture~\ref{SS conj1}, together with
Conjecture~\ref{conj 2}. The former gives a weight-$2$ mod-$p$
eigenform at the Serre conductor of $\bar\rho_{E,p}$, which ranges
over finitely many levels. Using the cohomological tools of
\c{S}eng\"un--Siksek~\cite{SS18}, we lift this eigenform to
characteristic zero for sufficiently large $p$; this step gives
no explicit bound. After establishing rationality of the Hecke
eigenvalues, Conjecture~\ref{conj 2} supplies an elliptic curve,
or possibly a fake elliptic curve when $K$ is totally complex.
The latter is excluded because its residual inertia image at
$\mfP$ has order at most $24$, whereas that of the Frey curve
has order divisible by $p$.

Over totally real fields, modularity in
Theorem~\ref{totally real unrestricted criterion}(1) follows from
\cite[Theorem~3.2]{S26}. In part~(2), $v_{\mfP}(j_E)$ tends
uniformly to $-\infty$ as $p$ grows, excluding the finitely many
non-modular $j$-invariants supplied by \cite[Theorem~5]{FS15}.
Remark~\ref{remark totally real level lowering} then supplies level lowering.
After establishing rationality of the Hecke eigenvalues, the
local inertia condition permits the application of
\cite[Theorem~3.6]{S26}, supplying the auxiliary elliptic curve
unconditionally. Thus neither Conjecture~\ref{conj 1} nor
Conjecture~\ref{conj 2} is required for either part of
Theorem~\ref{totally real unrestricted criterion}.

In both cases, the residual comparison yields an elliptic curve
$E'/K$ with full rational $2$-torsion, good reduction outside
$S_K'\cup\{\mathfrak m\}$, and $v_{\mfP}(j_{E'})<0$.
An inertia comparison at $\mathfrak m$ forces potentially good
reduction there, so $j_{E'}\in\mcO_{S_K'}$.
A Legendre parameter of $E'$ then gives an $S_K'$-unit solution
$\lambda+\mu=1$. The assumed valuation bound forces
$v_{\mfP}(j_{E'})\geq0$, yielding the contradiction.

For imaginary quadratic fields $K=\Q(\sqrt{-d})$ with $d\geq5$
squarefree and $-d\not\equiv1\pmod8$, there is a unique prime
above $2$. We determine the corresponding $S_K$-unit group
$\{\pm2^n:n\in\Z\}$ and solve the $S_K$-unit equation explicitly,
covering both ramified and inert primes above $2$.
For real quadratic fields, we use Freitas--Siksek's density theorem
\cite[Theorem~4]{FS15} for fields whose $S_d$-unit equation has
only the irrelevant solutions.

\subsection{Structure of the article}
Section~\ref{section for preliminary} recalls the modularity
conjectures and the lifting of mod-$p$ eigenforms.
In \S\ref{steps to prove main results}, we establish the local
calculations and a common modular argument for both main criteria.
Section~\ref{section for loc criteria} treats the imaginary
quadratic applications and the density calculation.
Section~\ref{section real quadratic proofs} proves the real
quadratic applications and the density-one result.
\section*{Acknowledgements}
During the revision of this article, the author used ChatGPT to assist
with checking local calculations and mathematical arguments, clarifying
proofs, and comparing hypotheses with cited results. The author
independently reviewed the resulting text and arguments and takes full
responsibility for the mathematical content of the article.

\section{Preliminaries}
\label{section for preliminary}
In this section, we recall the definitions of eigenforms for \(\GL_2\) and the lifting of mod-\(p\) eigenforms. We then formulate conjectures concerning the modularity of mod-\(p\) Galois representations and the Eichler--Shimura correspondence over number fields.
Throughout this article, we fix a number field $K$ not necessarily totally real. Let
$
G_K:=\Gal(\overline{K}/K)$
denote the absolute Galois group of $K$.
\subsection{Eigenforms for \texorpdfstring{$\GL_2$}{GL2} over number fields}
\label{subsection eigenforms over number fields}
Let $K$ be a number field and let $\mcO_K$
denote its ring of integers. In this subsection, we recall the
notions of complex and mod-$p$ eigenforms over $K$; we refer to
\cite[\S2]{SS18} for further details.

Let $(r,s)$ denote the signature of $K$, where $r$ is the number of
real embeddings of $K$ and $s$ is the number of pairs of complex
conjugate embeddings. Thus, $
[K:\Q]=r+2s.$ Let $\widehat{\mcO}_K$ denote the profinite completion of $\mcO_K$,
and let $\A_K$ and $\A_K^\infty$ denote the ring of ad\`eles and the
ring of finite ad\`eles of $K$, respectively. We write
$\mbH_2^\pm$ for the union of the upper and lower half-planes and
$\mbH_3$ for hyperbolic $3$-space. Set
\[
X:=(\mbH_2^\pm)^r\times\mbH_3^s.
\]
The group $\GL_2(K)$ acts on $X$ through the natural embedding
\[
\GL_2(K)
\hookrightarrow
\GL_2(K\otimes_{\Q}\R)
\simeq
\GL_2(\R)^r\times\GL_2(\C)^s.
\]

For a non-zero integral ideal $\mfn$ of $\mcO_K$, define
\[
U_0(\mfn)
:=
\left\{
\gamma\in\GL_2(\widehat{\mcO}_K):
\gamma\equiv
\bsmat{\ast}{0}{\ast}{\ast}
\pmod{\mfn}
\right\},
\]
and
\[
Y_0(\mfn)
:=
\GL_2(K)\backslash
\left(
\bigl(\GL_2(\A_K^\infty)/U_0(\mfn)\bigr)\times X
\right).
\]

For $i\in\{0,\ldots,2r+3s\}$, consider the $i$-th cohomology group
$
H^i(Y_0(\mfn),\C).
$
For every prime ideal $\mfq$ of $\mcO_K$ satisfying
$\mfq\nmid\mfn$, let $T_{\mfq}$ denote the corresponding Hecke
operator acting on $H^i(Y_0(\mfn),\C)$. These Hecke operators commute
with one another. We denote by
$
\mathbb T_{\C}^{i}(\mfn)
$
the commutative $\Z$-algebra generated by the operators $T_{\mfq}$,
for $\mfq\nmid\mfn$, inside
$
\text{End}_{\C}\bigl(H^i(Y_0(\mfn),\C)\bigr).
$

\begin{dfn}
	\label{definition complex eigenform}
	Let $i\in\{0,\ldots,2r+3s\}$ and let $\mfn$ be a non-zero ideal of $\mcO_K$. A \emph{complex eigenform of weight
		$2$, degree $i$, and level $\mfn$ over $K$} is a ring
	homomorphism
	$
	f:\mathbb T_{\C}^{i}(\mfn)\longrightarrow\C.
	$
\end{dfn}

For a complex eigenform $f$, the Hecke eigenvalues $f(T_{\mfq})$ are
algebraic integers and generate a number field, which we denote by
$\Q_f$ and call the \emph{Hecke field} (or \emph{coefficient field})
of $f$.
A complex eigenform $f$ of level $\mfn$ is called \emph{trivial} if
$
f(T_{\mfq})=\pm\bigl(N_{K/\Q}(\mfq)+1\bigr)
$
for every prime ideal $\mfq\nmid\mfn$.
Two complex eigenforms $f$ and $g$, possibly of different degrees
and levels, are said to be \emph{equivalent} if
$
f(T_{\mfq})=g(T_{\mfq})
$
for all but finitely many prime ideals $\mfq$.
Finally, a complex eigenform $f$ of level $\mfn$ is called
\emph{new} if it is not equivalent to any complex eigenform whose
level is a proper divisor of $\mfn$.

\subsection{Lifting mod-\texorpdfstring{$p$}{p} eigenforms}
\label{subsection lifting mod p eigenforms}
Let $\mfn$ be a non-zero integral ideal of $\mcO_K$, and let $p$ be
a rational prime which is unramified in $K$ and coprime to $\mfn$.
For $i\in\{0,\ldots,2r+3s\}$, the cohomology group
$
H^i(Y_0(\mfn),\overline{\F}_p)
$
is likewise equipped with Hecke operators, which we continue to
denote by $T_{\mfq}$, where $\mfq$ ranges over the prime ideals of
$\mcO_K$ satisfying
$
\mfq\nmid p\mfn.
$
These Hecke operators commute with one another and generate a
commutative Hecke algebra, which we denote by
$
\mathbb T_{\overline{\F}_p}^{i}(\mfn).
$
\begin{dfn}
	\label{definition mod p eigenform}
	Let $i\in\{0,\ldots,2r+3s\}$. A \emph{mod-$p$ eigenform of
		weight $2$, degree $i$, and level $\mfn$ over $K$} is a ring
	homomorphism
	$
	\theta:
	\mathbb T_{\overline{\F}_p}^{i}(\mfn)
	\longrightarrow
	\overline{\F}_p.
	$
\end{dfn}

We say that a mod-$p$ eigenform $\theta$ of degree $i$ and level
$\mfn$ \emph{lifts to a complex eigenform} if there exist a complex
eigenform $f$ of degree $i$ and level $\mfn$, together with a prime
ideal $\lambda$ of the Hecke field $\Q_f$ lying above $p$, such that
\[
\theta(T_{\mfq})
\equiv
f(T_{\mfq})
\pmod{\lambda}
\]
for every prime ideal $\mfq$ of $\mcO_K$ satisfying
$
\mfq\nmid p\mfn.
$

The following result guarantees that, for a fixed level, every
mod-$p$ eigenform lifts to a complex one once $p$ is
sufficiently large.

\begin{prop}[{\cite[Proposition~2.1]{SS18}}]
	\label{lifting to complex one}
	Let $\mfn$ be a non-zero integral ideal of $\mcO_K$. There exists
	a constant $B_{\mfn}>0$, depending only on $\mfn$, such that, for
	every prime
	$
	p>B_{\mfn},
	$
	every mod-$p$ eigenform of level $\mfn$ lifts to a complex
	eigenform of the same level.
\end{prop}
\subsection{Mod-\texorpdfstring{$p$}{p} Galois representations and modularity conjectures}
\label{subsection mod p Galois representations}

Let $K$ be a number field. Suppose that $K$ admits a real embedding
$
\sigma:K\hookrightarrow\mathbb R
$. For every extension
$
\widetilde{\sigma}:\overline{K}\hookrightarrow\mathbb C
$
of $\sigma$, the element
$
\widetilde{\sigma}^{-1}\circ\iota\circ\widetilde{\sigma}
\in G_K,
$
where $\iota$ denotes the usual complex conjugation on $\mathbb C$,
is called a \emph{complex conjugation} in $G_K$. A representation
$
\bar{\rho}:G_K\longrightarrow\GL_2(\overline{\F}_p)
$
is said to be \emph{odd} if
$
\det\bigl(\bar{\rho}(\mfc)\bigr)=-1
$
for every complex conjugation $\mfc\in G_K$. If $K$ has no real
embeddings, then we regard $\bar{\rho}$ as automatically odd.

We now recall a special case of Serre's modularity conjecture over
number fields concerning the modularity of two-dimensional mod-$p$
Galois representations.

\begin{conj}[{\cite[Conjecture~3.1]{SS18}}]
	\label{SS conj1}
    Let
    $\bar{\rho}:G_K\longrightarrow\GL_2(\overline{\F}_p)$
    be an odd, irreducible, and continuous representation with Serre
    conductor $\mfn$, i.e., $\mfn$ is the prime-to-$p$ part of its
    Artin conductor, and with trivial character, i.e., the
    prime-to-$p$ part of $\det(\bar{\rho})$ is trivial. Assume that
    $p$ is unramified in $K$ and that, for every prime $\mfp\mid p$,
    the restriction $\bar{\rho}|_{G_{K_{\mfp}}}$ arises from a
    finite flat group scheme over $\mcO_{K_{\mfp}}$.
    Then there exists a weight-$2$ mod-$p$ eigenform $\theta$ over
    $K$ of level $\mfn$ such that
    \[
        \Tr\bigl(\bar{\rho}(\Frob_{\mfq})\bigr)
        =\theta(T_{\mfq})
    \]
    for every prime ideal $\mfq\nmid p\mfn$.
\end{conj}
For an elliptic curve $E/K$ and a rational prime $p$, let
\[
\bar{\rho}_{E,p}:G_K
\longrightarrow
\Aut(E[p])
\simeq
\GL_2(\F_p)
\]
denote the mod-$p$ Galois representation arising from the natural
action of $G_K$ on the $p$-torsion subgroup $E[p]$ of $E$.

For our article, we need only the following special case of the
above conjecture.
\begin{conj}
	\label{modularity conj}
	\label{conj 1}
	Let $E/K$ be an elliptic curve over a number field, and let
	$p\geq5$ be a rational prime unramified in $K$. Suppose that
	$\bar{\rho}_{E,p}$ is absolutely irreducible and that, for every
	prime $\mfp\mid p$, the curve $E$ is semistable at $\mfp$ and
	$
	p\mid v_{\mfp}(\Delta_{\mfp}(E)),$
	where $\Delta_{\mfp}(E)$ denotes the minimal
	discriminant of $E$ at $\mfp$. Let $\mfn$ be the Serre conductor of
	$\bar{\rho}_{E,p}$. Then there exists a weight-$2$ mod-$p$ eigenform
	$\theta$ over $K$ of level $\mfn$ such that
	$
	\Tr\bigl(\bar{\rho}_{E,p}(\Frob_{\mfq})\bigr)
	=\theta(T_{\mfq})
	$
	for every prime ideal $\mfq\nmid p\mfn$.
\end{conj}

The local hypotheses imply that $\bar{\rho}_{E,p}$ is finite flat
at every prime above $p$; see the criterion used in
\cite[Lemma~5.4]{SS18}. At primes of good reduction, the minimal
discriminant exponent is zero, so the divisibility condition is
automatic. Moreover, the Weil pairing gives
$\det(\bar{\rho}_{E,p})=\chi_p$, where $\chi_p$ is the mod-$p$
cyclotomic character. Thus oddness and the character condition in
Conjecture~\ref{SS conj1} are automatic. The formulation above is
therefore a special case of that conjecture.

\begin{remark}
    Caraiani--Newton~\cite[Theorem~1.1]{CN} prove that every
    elliptic curve over an imaginary quadratic field $K$ is modular
    if $X_0(15)(K)$ is finite; this includes $K=\Q(\sqrt{-d})$
    for $d=1,2,3,5$. Modularity gives an eigenform at the conductor
    $\mathfrak N_E$ of $E$, whereas Conjecture~\ref{conj 1}
    requires a weight-$2$ mod-$p$ eigenform at the Serre conductor
    $\mfn$ of $\bar\rho_{E,p}$. These levels may differ: at a
    multiplicative prime $\mfq\nmid p$,
    \[
        p\mid v_{\mfq}(\Delta_{\mfq}(E))
        \quad\Longrightarrow\quad
        \mfq\mid\mathfrak N_E,\qquad \mfq\nmid\mfn.
    \]
    Passing to $\mfn$ requires additional level lowering, which
    their theorem does not provide; see~\cite[p.~3]{CN}.
    Conjecture~\ref{conj 1} is therefore retained to obtain
    eigenforms at the fixed finite set of levels supplied by
    Lemma~\ref{Serre's conductor}.
    For totally real $K$ and modular $E/K$, Conjecture~\ref{conj 1}
    follows from known level-lowering results, as explained in
    Remark~\ref{remark totally real level lowering} below.
\end{remark}

\begin{remark}
    \label{remark totally real level lowering}
    Suppose that $K$ is totally real and $E/K$ is modular.
    Modularity gives a Hilbert eigenform of parallel weight $2$
    at the conductor $\mathfrak N_E$ of $E$.
    Since $p\geq5$ is unramified in $K$, we have
    $e(\mfp/p)=1<p-1$ for $\mfp\mid p$ and
    $\Q(\zeta_p)^+\not\subseteq K$. Thus the hypotheses of
    Conjecture~\ref{conj 1} allow us to apply
    \cite[Theorem~7]{FS15}, obtaining a Hilbert eigenform $f$
    of parallel weight $2$ with
    $\bar\rho_{E,p}\simeq\bar\rho_{f,\lambda}$ for some $\lambda\mid p$,
    at level
    \[
        \mathfrak N_p=\mathfrak N_E\Big/
        \prod_{\substack{\mfq\parallel\mathfrak N_E\\
        p\mid v_{\mfq}(\Delta_{\mfq}(E))}}\mfq.
    \]
    In fact, $\mathfrak N_p=\mfn$. At a multiplicative prime
    $\mfq\nmid p$, the representation $\bar\rho_{E,p}$ is
    unramified exactly when $p\mid v_{\mfq}(\Delta_{\mfq}(E))$;
    these are precisely the factors removed above. At additive
    primes away from $p$, reduction modulo $p\geq5$ preserves
    the conductor exponent, while good primes contribute nothing.
    Finally, semistability and the discriminant divisibility
    assumption remove every factor above $p$. Hence reducing $f$
    modulo $\lambda$ gives the eigenform $\theta$ of level $\mfn$
    required by Conjecture~\ref{conj 1}.
\end{remark}

We next recall an Eichler--Shimura-type conjecture over arbitrary
number fields. Recall that a \emph{fake elliptic curve} over $K$ is
a simple abelian surface $A/K$ whose endomorphism algebra
$
\End_K(A)\otimes_{\mathbb Z}\mathbb Q
$
is an indefinite division quaternion algebra over $\mathbb Q$.
The following conjecture relates rational weight-$2$ eigenforms to
elliptic curves or fake elliptic curves over $K$ and will also be
used in the proofs of our main results.

\begin{conj}[{\cite[Conjecture~4.1]{SS18}}]
	\label{conj 2}
	Let $f$ be a non-trivial new complex eigenform of weight $2$ over
	$K$, of level $\mfn$, and suppose that
	$
	\Q_f=\Q.
	$
	If $K$ has at least one real place, then there exists an elliptic
	curve $E_f/K$ of conductor $\mfn$ such that
	\begin{equation}
		\label{eqn for ES1}
		\#E_f(\mcO_K/\mfq)
		=
		1+N_{K/\Q}(\mfq)-f(T_{\mfq})
	\end{equation}
	for every prime ideal $\mfq\nmid\mfn$.

	If $K$ is totally complex, then there exists either an elliptic
	curve $E_f/K$ of conductor $\mfn$ satisfying
	\eqref{eqn for ES1}, or a fake elliptic curve $A_f/K$ of conductor
	$\mfn^2$ such that
	\begin{equation}
		\label{eqn for ES2}
		\#A_f(\mcO_K/\mfq)
		=
		\left(
		1+N_{K/\Q}(\mfq)-f(T_{\mfq})
		\right)^2
	\end{equation}
	for every prime ideal $\mfq\nmid\mfn$.
\end{conj}

\begin{remark}
    \label{remark totally real ES}
    When $K$ is totally real, Conjecture~\ref{conj 2} specializes
    to the Eichler--Shimura conjecture for Hilbert newforms of
    parallel weight $2$ with rational eigenvalues; see~\cite[Conjecture~1]{FS15}.
    Indeed, \eqref{eqn for ES1} gives
    $a_{\mfq}(E_f)=f(T_{\mfq})$ for every $\mfq\nmid\mfn$,
    identifying the good Euler factors. Chebotarev density and
    Brauer--Nesbitt identify the associated semisimple $\ell$-adic
    representations; local--global compatibility then gives
    equality of all local Euler factors, and hence
    $L(E_f,s)=L(f,s)$.
    Moreover, Conjecture~\ref{conj 2} holds over every totally real
    field of odd degree; see~\cite[Theorem~8]{FS15}.
\end{remark}

\section{Local analysis and proofs of the main results}
\label{steps to prove main results}

We prove both parts of
Theorem~\ref{main result1 for Ax^p+By^p=Cz^p} using a common
modular argument. The two local calculations below supply the
reduction type and the image of inertia properties needed to construct the
auxiliary elliptic curve. For an elliptic curve $E/K$, we denote by $\Delta_E$ and
$j_E$ the discriminant and the $j$-invariant of $E$, respectively.

For a non-trivial solution $(a,b,c)\in K^3$, put
$\mathfrak g=a\mcO_K+b\mcO_K+c\mcO_K$.
The fractional ideal $\mathfrak g$ need not be principal, so a common scaling
need not produce a primitive integral solution. Instead, choose
$\mathfrak m\in\mathcal H$ with $[\mathfrak m]=[\mathfrak g]$,
where $\mathcal H$ is a fixed set of prime ideals outside $S_K'$ representing each ideal class of $\operatorname{Cl}(K)$, as in the normalization convention
preceding Theorem~\ref{main result1 for Ax^p+By^p=Cz^p}. Since
$\mathfrak m\mathfrak g^{-1}$ is principal, write
$\mathfrak m\mathfrak g^{-1}=(\xi)$ with $\xi\in K^\times$.
Then
\[
    (\xi a)\mcO_K+(\xi b)\mcO_K+(\xi c)\mcO_K
    =(\xi)\mathfrak g=\mathfrak m\subseteq\mcO_K.
\]
Hence $\xi a,\xi b,\xi c\in\mcO_K$.
Moreover,
\[
    A(\xi a)^p+B(\xi b)^p-C(\xi c)^p
    =\xi^p(Aa^p+Bb^p-Cc^p)=0.
\]
Thus $(\xi a,\xi b,\xi c)$ is a normalized integral solution;
see \cite[\S5]{SS18} and \cite[Lemma~3.7]{S26}.
We use this representative throughout the local calculations
and again denote it by $(a,b,c)$. The common-ideal identity gives
\[
    \min\{v_{\mfq}(a),v_{\mfq}(b),v_{\mfq}(c)\}
    =v_{\mfq}(\mathfrak m)
    =\begin{cases}1,&\mfq=\mathfrak m,\\0,&\mfq\neq\mathfrak m.\end{cases}
\]
In particular, the minimum is zero at every prime ideal $\mfq \in S_K'$,
since $\mathfrak m\notin S_K'$; this is a consequence of scaling,
not an additional hypothesis.
For each $x\in\{a,b,c\}$, one has
$v_{\mfP}(\xi x)=v_{\mfP}(\xi)+v_{\mfP}(x)$.
Hence common scaling preserves the strict inequality defining
$W_{K,\mfP}$.
The finiteness of $\mathcal H$ makes the bounds at $\mathfrak m$
uniform in the solution.

\subsection{The Frey curve}
	Let $(a,b,c)\in K^3$ be a non-trivial solution of the generalized
	Fermat equation
$	Aa^p+Bb^p=Cc^p.$
We associate to $(a,b,c)$ the Frey elliptic curve
\begin{equation}
	\label{Frey curve for x^2=By^p+Cz^p of Type I}
	E:=E_{a,b,c}:
	y^2=x(x-Aa^p)(x+Bb^p).
\end{equation}
A direct computation gives
\begin{align}
	\label{formula for c_4}
	c_4(E)
	&=
	2^4\left(A^2a^{2p}+BCb^pc^p\right) \notag\\
	&=
	2^4\left(B^2b^{2p}+ACa^pc^p\right) \notag\\
	&=
	2^4\left(C^2c^{2p}-ABa^pb^p\right),
\end{align}
and
$	\Delta_E
	=
	2^4A^2B^2C^2(abc)^{2p}.$
Consequently,
$	j_E
	=
	2^8
	\frac{
		\left(A^2a^{2p}+BCb^pc^p\right)^3
	}{
		A^2B^2C^2(abc)^{2p}
	}.$

\subsection{Conductor computation}
\label{subsection conductor computation}

For a prime ideal $\mfq$ of $\mcO_K$, let $\Delta_{\mfq}(E)$ denote the
minimal discriminant of $E$ at $\mfq$. We first determine the
reduction of a normalized Frey curve at primes outside
$S_K'\cup\{\mathfrak m\}$. The following lemma also
gives the properties of $\bar{\rho}_{E,p}$ that will be needed in
the application of Conjecture~\ref{conj 1}.

\begin{lem}
	\label{Serre's conductor}
    Let $(a,b,c)\in\mcO_K^3$ be a non-trivial normalized solution of
    \eqref{Ax^p+By^p=Cz^p}, and let $E$ be the Frey curve
    \eqref{Frey curve for x^2=By^p+Cz^p of Type I}.
    Suppose that $p\geq5$ is a rational prime such that no prime
    of $K$ above $p$ belongs to $S_K'\cup\mathcal H$.
    Then the following statements hold:
    \begin{enumerate}
		\item For every prime $\mfq\notin S_K'\cup\{\mathfrak m\}$, the curve $E$ is
		minimal and semistable at $\mfq$, and
		$
		p\mid v_{\mfq}(\Delta_E).
		$

		\item The Serre conductor $\mfn$ of $\bar{\rho}_{E,p}$ is
		supported on $S_K'\cup\{\mathfrak m\}$ and belongs to a finite set of ideals
		depending only on $K$, $A$, $B$, and $C$.
	\end{enumerate}
\end{lem}

\begin{proof}
	Let $\mfq\notin S_K'\cup\{\mathfrak m\}$. Since
	$
	\Delta_E
	=
	2^4A^2B^2C^2(abc)^{2p},$
	if $\mfq\nmid\Delta_E$, then $E$ has good reduction at $\mfq$.
	In this case,
	$
	v_{\mfq}(\Delta_E)=0,$
	and hence
	$
	p\mid v_{\mfq}(\Delta_E).$
	Suppose now that $\mfq\mid\Delta_E$. Since
	$\mfq\notin S_K'\cup\{\mathfrak m\}$, we have
	$
	\mfq\nmid2ABC.
	$
	Therefore $\mfq\mid abc$. Since $\mfq\neq\mathfrak m$, the prime $\mfq$
	cannot divide all three of $a,b,c$. If $\mfq$ divided two of them,
	then the equation $Aa^p+Bb^p=Cc^p$, together with $\mfq\nmid ABC$,
	would force $\mfq$ to divide the third. Hence $\mfq$ divides
	exactly one of $a,b,c$.

	We claim that
	$
	v_{\mfq}(c_4(E))=0.$
	Indeed, according as $\mfq$ divides $a$, $b$, or $c$, we use,
	respectively, one of the identities in \eqref{formula for c_4}.
	Since $\mfq\nmid2ABC$ and $\mfq$ divides exactly one of
	$a,b,c$, it follows in each case that
	$
	\mfq\nmid c_4(E).$
	Hence the given model is minimal at $\mfq$ and $E$ has
	multiplicative reduction there. Moreover,
	$
	v_{\mfq}(\Delta_E)
	=
	2p\,v_{\mfq}(abc),$
	so that
	$
	p\mid v_{\mfq}(\Delta_E).$
	This proves (1).

	It remains to determine the support of the Serre conductor.
	Let $\mfq\notin S_K'\cup\{\mathfrak m\}$ with $\mfq\nmid p$. If $E$ has good reduction at $\mfq$, then
	the N\'eron--Ogg--Shafarevich criterion implies that
	$\bar{\rho}_{E,p}$ is unramified at $\mfq$. If $E$ has
	multiplicative reduction at $\mfq$, then
	$
	p\mid v_{\mfq}(\Delta_E),
	$
	and the standard description of the mod-$p$ representation of a
	Tate curve again shows that $\bar{\rho}_{E,p}$ is unramified at
	$\mfq$. Therefore
	$
	\mfq\nmid\mfn
	\text{ for every  }\mfq\notin S_K'\cup\{\mathfrak m\}\text{ with }\mfq\nmid p.
	$
	Thus $\mfn$ is supported on $S_K'\cup\{\mathfrak m\}$.

	Finally, the Serre conductor $\mfn$ divides the conductor
	$\mathfrak N_E$ of $E$. By the standard bounds for the local
	conductor exponent of an elliptic curve
	(cf.~\cite[Theorem~IV.10.4]{S94}), we have
	\[
	v_{\mfq}(\mfn)
	\leq
	v_{\mfq}(\mathfrak N_E)
	\leq
	\begin{cases}
		2+6v_{\mfq}(2), & \mfq\mid2,\\[2mm]
		2+3v_{\mfq}(3), & \mfq\nmid2.
	\end{cases}
	\]
	Since $S_K'\cup\{\mathfrak m\}\subseteq S_K'\cup\mathcal H$ and the latter set is
    finite and depends only on $K,A,B,C$, there are
	only finitely many possibilities for $\mfn$, depending only on
	$K,A,B,C$. This proves (2).
\end{proof}

\subsection{Image of inertia}
\label{subsection image of inertia}
For a prime ideal $\mfP$ of $\mcO_K$, let $I_{\mfP}$ denote
the inertia subgroup of $G_K$ at $\mfP$.
In this subsection, we study the image of inertia under the residual
Galois representation
$
\bar{\rho}_{E,p},
$
which will be useful for determining the reduction type of the Frey
curve at primes $\mfP\in S_K$. We begin by recalling the following
standard criterion.

\begin{lem}[{\cite[Lemma~5.1]{SS18}}]
	\label{criteria for potentially multiplicative reduction}
	Let $E/K$ be an elliptic curve, let $p\geq5$ be a rational prime,
	and let $\mfq$ be a prime of $K$ such that $\mfq\nmid p$. Then:
	\begin{enumerate}
		\item If $v_{\mfq}(j_E)\geq0$, then
		$
		\#\bar{\rho}_{E,p}(I_{\mfq})\mid24.
		$

		\item Suppose that $v_{\mfq}(j_E)<0$. Then:
		\begin{enumerate}
			\item if $p\nmid v_{\mfq}(j_E)$, then
			$
			\#\bar{\rho}_{E,p}(I_{\mfq})=p
			\quad\text{or}\quad 2p;
			$
			\item if $p\mid v_{\mfq}(j_E)$, then
			$
			\#\bar{\rho}_{E,p}(I_{\mfq})=1
			\quad\text{or}\quad 2.
			$
		\end{enumerate}
	\end{enumerate}
\end{lem}

For $\mfP\in S_K$ and $A,B,C\in\mcO_K\setminus\{0\}$, recall that
$
	\alpha_{\mfP}:=8v_{\mfP}(2)+v_{\mfP}(BCA^{-2}),\
	\beta_{\mfP}:=8v_{\mfP}(2)+v_{\mfP}(ACB^{-2}),\
	\gamma_{\mfP}:=8v_{\mfP}(2)+v_{\mfP}(ABC^{-2}),$
and define
\begin{equation}
	\label{bound for pot mult red}
	M_{\mfP}(A,B,C)
	:=
	\max\left\{
	v_{\mfP}(ABC),
	|\alpha_{\mfP}|,
	|\beta_{\mfP}|,
	|\gamma_{\mfP}|
	\right\}.
\end{equation}

The following lemma describes the reduction of the Frey curve at
the primes in $S_K$ and, in particular, determines when the image
of inertia has order divisible by $p$.

\begin{lem}
	\label{reduction on T and S}
	Let $\mfP\in S_K$, and let
    $(a,b,c)\in W_{K,\mfP}$ be a non-trivial normalized solution
    of \eqref{Ax^p+By^p=Cz^p}.
    Suppose that
	$
	p>M_{\mfP}(A,B,C),$
	and let $E$ be the Frey curve given by
	\eqref{Frey curve for x^2=By^p+Cz^p of Type I}. Then $E$ has
	potentially multiplicative reduction at $\mfP$. More precisely:
	\begin{enumerate}
		\item if $\mfP\mid a$, then
		$
		v_{\mfP}(j_E)
		=
		\alpha_{\mfP}
		-2p\,v_{\mfP}(a)<0;
		$
		\item if $\mfP\mid b$, then
		$
		v_{\mfP}(j_E)
		=
		\beta_{\mfP}
		-2p\,v_{\mfP}(b)<0;
		$
		\item if $\mfP\mid c$, then
		$
		v_{\mfP}(j_E)
		=
		\gamma_{\mfP}
		-2p\,v_{\mfP}(c)<0.
		$
	\end{enumerate}
	Furthermore, in the respective cases,
	$
	p\mid\#\bar{\rho}_{E,p}(I_{\mfP})
	$
	if and only if
	$
	\alpha_{\mfP}\neq0,\
	\beta_{\mfP}\neq0, \
	\gamma_{\mfP}\neq0,
	$
	respectively.
\end{lem}

\begin{proof}
	Recall that 	$j_E
	=
	2^8
	\frac{
		\left(A^2a^{2p}+BCb^pc^p\right)^3
	}{
		A^2B^2C^2(abc)^{2p}
	}.$ By Remark~\ref{remark for W_K}, the prime $\mfP$ divides exactly
	one of $a,b,c$. We treat the three possibilities separately.

	Suppose first that $\mfP\mid a$. Then $\mfP\nmid bc$. Since
	$
	p>M_{\mfP}(A,B,C)\geq v_{\mfP}(ABC)
	\geq v_{\mfP}(BC),
	$
	and $v_{\mfP}(a)\geq1$, we have
	$
	2p\,v_{\mfP}(a)>v_{\mfP}(BC).
	$
	Therefore, using
	$
	c_4(E)
	=
	2^4(A^2a^{2p}+BCb^pc^p),
	$
	we obtain
	$
	v_{\mfP}
	(A^2a^{2p}+BCb^pc^p)
	=
	v_{\mfP}(BC).
	$
	Hence
	\begin{align*}
		v_{\mfP}(j_E)
		&=
		8v_{\mfP}(2)
		+3v_{\mfP}(BC)
		-2v_{\mfP}(ABC)
		-2p\,v_{\mfP}(a)\\
		&=
		8v_{\mfP}(2)
		+v_{\mfP}(BCA^{-2})
		-2p\,v_{\mfP}(a)\\
		&=
		\alpha_{\mfP}
		-2p\,v_{\mfP}(a).
	\end{align*}
	Since
	$
	p>|\alpha_{\mfP}|
	$
	and $v_{\mfP}(a)\geq1$, it follows that
	$
	v_{\mfP}(j_E)<0.
	$
	Thus $E$ has potentially multiplicative reduction at $\mfP$.

	Moreover,
	$
	v_{\mfP}(j_E)
	\equiv\alpha_{\mfP}\pmod p.
	$
	Since $p>|\alpha_{\mfP}|$, we have
	$
	p\mid v_{\mfP}(j_E)
	\quad\Longleftrightarrow\quad
	\alpha_{\mfP}=0.
	$
	Consequently, if $\alpha_{\mfP}\neq0$, then
	$
	p\nmid v_{\mfP}(j_E),
	$
	and Lemma~
	\ref{criteria for potentially multiplicative reduction}(2)
	implies $
	p\mid\#\bar{\rho}_{E,p}(I_{\mfP}).$
	On the other hand, if $\alpha_{\mfP}=0$, then
	$v_{\mfP}(j_E)
	=
	-2p\,v_{\mfP}(a),$
	so that
$	p\mid v_{\mfP}(j_E).$
	By Lemma~
	\ref{criteria for potentially multiplicative reduction}(2),
$	\#\bar{\rho}_{E,p}(I_{\mfP})\in\{1,2\},$
	and therefore
$	p\nmid\#\bar{\rho}_{E,p}(I_{\mfP}).$

	The cases $\mfP\mid b$ and $\mfP\mid c$ are entirely analogous.
	Indeed, if $\mfP\mid b$, then
	$
	v_{\mfP}(j_E)
	=
	\beta_{\mfP}-2p\,v_{\mfP}(b),
	$
	whereas if $\mfP\mid c$, then
	$
	v_{\mfP}(j_E)
	=
	\gamma_{\mfP}-2p\,v_{\mfP}(c).
	$
	The same argument yields the desired conclusions.
\end{proof}

\begin{lem}
    \label{local reduction without divisibility}
    Let $\mfP\in S_K$ and put $e=v_{\mfP}(2)$. Suppose $v_{\mfP}(A)=v_{\mfP}(B)=m$ and $
    v_{\mfP}(C)=m+k$ for $ k>4e$.
    Let $(a,b,c)$ be a non-trivial normalized solution
    of \eqref{Ax^p+By^p=Cz^p} with
    $
        p>\max\{k,\,2(k-4e),\,5\}.$
    For the associated Frey curve $E$, we have
    $v_{\mfP}(a)=v_{\mfP}(b)=0$ and
    \begin{equation}
        \label{j valuation without divisibility}
        v_{\mfP}(j_E)=8e-2k-2p\,v_{\mfP}(c)<0.
    \end{equation}
    Moreover,
    $p\nmid v_{\mfP}(j_E)$ and
    $p\mid\#\bar{\rho}_{E,p}(I_{\mfP})$.
\end{lem}

\begin{proof}
    Suppose that $\mfP\mid a$.
    If $\mfP\nmid b$, then the term $Bb^p$ has valuation $m$,
    strictly smaller than those of $Aa^p$ and $Cc^p$, which
    contradicts $Aa^p+Bb^p=Cc^p$.
    Hence $\mfP\mid b$. Since
    $\min\{v_{\mfP}(a),v_{\mfP}(b),v_{\mfP}(c)\}=0$,
    it follows that $\mfP\nmid c$.
    But then
    \[
        v_{\mfP}(Aa^p+Bb^p)\geq m+p>m+k
            =v_{\mfP}(Cc^p),
    \]
    another contradiction.
    Thus $v_{\mfP}(a)=0$, and the same argument gives
    $v_{\mfP}(b)=0$.

   This gives $v_{\mfP}(c)\geq0$.
    The terms $A^2a^{2p}$ and $BCb^pc^p$ have respective
    valuations $2m$ and $2m+k+pv_{\mfP}(c)$.
    The formulas for the Frey invariants therefore give
    \[
        v_{\mfP}(c_4(E))=4e+2m,\qquad
        v_{\mfP}(\Delta_E)=4e+6m+2k+2pv_{\mfP}(c).
    \]
    Since $j_E=c_4(E)^3/\Delta_E$, we get $v_{\mfP}(j_E)=8e-2k-2p\,v_{\mfP}(c)$ and  $v_{\mfP}(j_E) \equiv 8e-2k \pmod p$.
    The bound $p>2(k-4e)$ consequently gives  $v_{\mfP}(j_E) <0$ and
    $p\nmid v_{\mfP}(j_E)$.
    Finally, by Lemma~\ref{criteria for potentially multiplicative reduction}(2), we have $p\mid\#\bar{\rho}_{E,p}(I_{\mfP})$.
\end{proof}
\subsection{Surjectivity of the mod-\texorpdfstring{$p$}{p} Galois representation}
\label{subsection surjectivity mod p}

In order to apply Conjecture~\ref{modularity conj}, we need to establish
the absolute irreducibility of the residual Galois representation
$\bar{\rho}_{E,p}$. In fact, for sufficiently large $p$, we prove the
stronger assertion that $\bar{\rho}_{E,p}$ is surjective.
We first recall the following irreducibility criterion.

\begin{thm}[{\cite[Proposition~6.1]{SS18}}]
	\label{irreducibility of mod $P$ representation}
	Let $L$ be a Galois number field and let $\mfq$ be a prime of $L$.
	Then there exists a constant $C_{L,\mfq}>0$ with the following
	property. Let $p>C_{L,\mfq}$ be a rational prime and let $E/L$ be
	an elliptic curve which is semistable at every prime $\mfp\mid p$
	and has potentially multiplicative reduction at $\mfq$. Then
	$
	\bar{\rho}_{E,p}:G_L\longrightarrow\GL_2(\F_p)
	$
	is irreducible.
\end{thm}

Combining this result with the standard inertia criterion of
Lemma~\ref{criteria for potentially multiplicative reduction}, we obtain the following.

\begin{lem}
	\label{surjective of mod p GR}
    Let $A,B,C\in\mcO_K\setminus\{0\}$ and let $\mfP\in S_K$.
    There exists a constant $D_{K,A,B,C}>0$ with the following
    property. Let $(a,b,c)\in\mcO_K^3$ be a non-trivial
    normalized solution of \eqref{Ax^p+By^p=Cz^p} with
    $p>D_{K,A,B,C}$, and let $E$ be its Frey curve.
    If
    \begin{equation}
        \label{local conditions for modular argument}
        v_{\mfP}(j_E)<0,\qquad p\nmid v_{\mfP}(j_E),
    \end{equation}
    then
    $\bar{\rho}_{E,p}(G_K)=\GL_2(\F_p)$.
    In particular, $\bar{\rho}_{E,p}$ is absolutely irreducible.
\end{lem}

\begin{proof}
    Let $L$ be the Galois closure of $K/\Q$, and choose
    $\mfq\mid\mfP$ in $L$. Take $p>\max\{5,[L:K]\}$ unramified
    in $K$, with no prime above $p$ in $S_K'\cup\mathcal H$.
    Lemma~\ref{Serre's conductor}(1) gives semistability above $p$,
    which is preserved over $L$. The local conditions
    \eqref{local conditions for modular argument} imply that $E$
    has potentially multiplicative reduction at $\mfP$, and hence
    at $\mfq$, with $p\mid\#\bar\rho_{E,p}(I_{\mfP})$ by
    Lemma~\ref{criteria for potentially multiplicative reduction}(2).
    Since $[I_{\mfP}:I_{\mfq}]\leq[L:K]<p$, restriction preserves
    this divisibility. Thus
    \begin{equation}
        \label{p divides inertia over L}
        p\mid\#\bar\rho_{E,p}(I_{\mfq})
        \quad\text{and}\quad p\mid\#\bar\rho_{E,p}(G_L).
    \end{equation}

    By Theorem~\ref{irreducibility of mod $P$ representation},
    $H:=\bar\rho_{E,p}(G_L)$ is irreducible for $p>C_{L,\mfq}$.
    Since $p\mid\#H$, the classification of subgroups of
    $\GL_2(\F_p)$ gives
    \begin{equation}
        \label{SL2 contained GK}
        \SL_2(\F_p)\subseteq H\subseteq\bar\rho_{E,p}(G_K).
    \end{equation}

    By the Weil pairing, $\det(\bar\rho_{E,p})=\chi_p$, the mod-$p$
    cyclotomic character. Since $p$ is unramified in $K$, whereas
    every non-trivial subfield of $\Q(\zeta_p)$ is ramified at $p$,
    we have $K\cap\Q(\zeta_p)=\Q$. Consequently,
    \begin{equation}
        \label{det full image}
        \det\bigl(\bar\rho_{E,p}(G_K)\bigr)
        =\operatorname{im}(\chi_p)=\F_p^\times.
    \end{equation}
    Combining \eqref{SL2 contained GK} and \eqref{det full image}
    gives $\bar\rho_{E,p}(G_K)=\GL_2(\F_p)$, and hence absolute
    irreducibility.
    Taking the maximum of $C_{L,\mfq}$ over the finitely many
    primes $\mfq\mid2$ of $L$, together with the preceding bounds,
    gives a constant depending only on $K,A,B,C$.
\end{proof}

\subsection{The auxiliary elliptic curve}
\label{subsection proof main theorem via conjectures}

The following theorem extracts the modular part of the argument: once
the required local conditions hold, it associates to any solution with
sufficiently large exponent an auxiliary elliptic curve with controlled
reduction and full rational $2$-torsion.
\begin{thm}
	\label{auxilary result x^2=By^p+Cz^p over W_K}
    Let $K$ be a number field for which Conjectures~\ref{conj 1}
    and~\ref{conj 2} hold, let $A,B,C\in\mcO_K\setminus\{0\}$,
    and fix $\mfP\in S_K$.
    There exists a constant $V=V_{K,A,B,C}>0$ such that, for every
    non-trivial normalized solution $(a,b,c)\in\mcO_K^3$ of
    \eqref{Ax^p+By^p=Cz^p} with $p>V$ whose Frey curve $E$
    satisfies \eqref{local conditions for modular argument},
    there is an elliptic curve $E'/K$ with the following properties:
    \begin{enumerate}
		\item $E'$ has good reduction away from $S_K'\cup\{\mathfrak m\}$ and potentially
        good reduction away from $S_K'$;

		\item $E'$ has full $2$-torsion over $K$, that is,
		$
		E'[2]\subseteq E'(K);$

		\item the residual Galois representations associated to $E$
		and $E'$ are isomorphic:
		$
		\bar{\rho}_{E,p}
		\sim
		\bar{\rho}_{E',p};$

		\item $E'$ has potentially multiplicative reduction at
		$\mfP$; equivalently,
		$
		v_{\mfP}(j_{E'})<0.$
	\end{enumerate}
\end{thm}

\begin{proof}[Proof of Theorem~\ref{auxilary result x^2=By^p+Cz^p over W_K}]
    We first pass from the Frey curve to an elliptic curve of bounded
    conductor. We then obtain full rational $2$-torsion and compare
    inertia at the primes where further reduction conditions are needed.
    Throughout, we increase a bound $V=V_{K,A,B,C}$ as necessary
    and take $p>V$.

    \emph{Step 1: level lowering.}
    Lemma~\ref{surjective of mod p GR} gives
    absolute irreducibility of $\bar\rho_{E,p}$. Take $p>24$
    unramified in $K$, with no prime above $p$ in $S_K'\cup\mathcal H$.
    By Lemma~\ref{Serre's conductor}(1), $E$ is minimal and
    semistable at every $\mfp\mid p$, with
    $
        v_{\mfp}(\Delta_{\mfp}(E))=2p\,v_{\mfp}(abc).
    $
    Thus Conjecture~\ref{conj 1} applies. Writing $\mfn$ for the
    Serre conductor, it gives a weight-$2$ mod-$p$ eigenform
    $\theta$ of level $\mfn$ with
    $\theta(T_{\mfq})=\Tr\bar\rho_{E,p}(\Frob_{\mfq})$
    for $\mfq\nmid p\mfn$.

    \emph{Step 2: rational eigenvalues.}
    By Lemma~\ref{Serre's conductor}(2), $\mfn$ is supported on
    $S_K'\cup\{\mathfrak m\}$ and ranges over a fixed finite set.
    Proposition~\ref{lifting to complex one} therefore gives a
    complex lift $f$ for sufficiently large $p$, with
    $
        f(T_{\mfq})\equiv\Tr\bar\rho_{E,p}(\Frob_{\mfq})
        \pmod{\lambda},
    $
    for all $\mfq\nmid p\mfn$ and for some $\lambda\mid p$ of $\Q_f$. Only finitely many such $f$
    occur. If $\Q_f\neq\Q$, then the argument of \cite[Lemma~7.2]{SS18},
    using semistability outside $S_K'\cup\mathcal H$, bounds $p$.
    Increasing $V$ gives $\Q_f=\Q$. Irreducibility excludes a
    trivial eigenform; replacing $f$ by an equivalent newform,
    we may assume its level is $\mfn'\mid\mfn$.

    \emph{Step 3: excluding fake elliptic curves.}
    Conjecture~\ref{conj 2} gives an elliptic curve $E_f/K$ of
    conductor $\mfn'$, or a fake elliptic curve $A_f/K$ of
    conductor $(\mfn')^2$. Fix one for each possible $f$.
    For sufficiently large $p$, the Frobenius congruences in Step~2
    and the cyclotomic determinants, together with Chebotarev density,
    Brauer--Nesbitt and absolute irreducibility, give
    \[
        \bar\rho_{E,p}\simeq
        \begin{cases}
            \bar\rho_{E_f,p},&\text{in the elliptic curve case},\\
            \bar\rho_{A_f,p},&\text{in the fake elliptic curve case},
        \end{cases}
    \]
    where $\bar\rho_{A_f,p}$ is the two-dimensional representation
    described in \cite[\S4]{SS18}.

    By \eqref{local conditions for modular argument} and
    Lemma~\ref{criteria for potentially multiplicative reduction}(2)(a),
    \begin{equation}
        \label{large inertia in auxiliary theorem}
        p\mid\#\bar\rho_{E,p}(I_{\mfP}).
    \end{equation}
    If the associated curve were $A_f$, then
    \cite[Theorem~4.2]{SS18} would give a totally ramified extension
    $L/K_{\mfP}$ of degree dividing $24$ over which $A_f$ has good
    reduction. Write $I_L$ for its inertia subgroup. Since
    $\mfP\in S_K$ and $p>24$, the N\'eron--Ogg--Shafarevich criterion
    makes $I_L$ act trivially on $A_f[p]$, hence on $\bar\rho_{A_f,p}$.
    Therefore
    \[
        \#\bar\rho_{A_f,p}(I_{\mfP})
        \mid[I_{\mfP}:I_L]
        =[L:K_{\mfP}]
        \mid24,
    \]
    contradicting \eqref{large inertia in auxiliary theorem} and
    the residual isomorphism. Thus the associated curve is $E_f$;
    since its conductor is $\mfn'\mid\mfn$, it has good reduction
    outside $S_K'\cup\{\mathfrak m\}$.

    \emph{Step 4: full rational $2$-torsion.}
    Since $E$ has full rational $2$-torsion and is semistable outside
    $S_K'\cup\mathcal H$, the argument of \cite[Lemma~7.4]{SS18},
    with auxiliary primes outside this set, shows that for sufficiently
    large $p$, either $E_f$ or a curve $2$-isogenous to it over $K$
    has full rational $2$-torsion. Since only finitely many $E_f$
    occur, the bound on $p$ is uniform.
    Choose such a curve $E'/K$. Good reduction is invariant under
    $K$-isogeny, and a $2$-isogeny induces an isomorphism on
    $p$-torsion as $G_K$-modules because $p>2$. Hence
    \[
        E'[2]\subseteq E'(K),\qquad
        \bar\rho_{E',p}\simeq\bar\rho_{E,p},
    \]
    and $E'$ has good reduction outside $S_K'\cup\{\mathfrak m\}$.

    \emph{Step 5: reduction at $\mfP$.}
    At $\mfP$, the residual isomorphism and
    \eqref{large inertia in auxiliary theorem} give
    $p\mid\#\bar\rho_{E',p}(I_{\mfP})$.
    By Lemma~\ref{criteria for potentially multiplicative reduction}, we get $v_{\mfP}(j_{E'})<0$,
    proving~(4).

    \emph{Step 6: potentially good reduction at $\mathfrak m$.}
    Let $\mathfrak m$ be the prime selected by the normalization, hence
    $
        a\mcO_K+b\mcO_K+c\mcO_K=\mathfrak m.
    $
    Hence $\min\{v_{\mathfrak m}(a),v_{\mathfrak m}(b),
    v_{\mathfrak m}(c)\}=1$. We show that
    $v_{\mathfrak m}(j_{E'})\geq0$.
    Put $t=\min\{v_{\mathfrak m}(a),v_{\mathfrak m}(b),
    v_{\mathfrak m}(c)\}$. Because $\mathfrak m\nmid ABC$, the equation
    $Aa^p+Bb^p=Cc^p$ forces at least two of the three valuations
    to equal $t$. If all three equal $t$, then the formulas for $c_4(E)$
    and $\Delta_E$ give $v_{\mathfrak m}(j_E)\geq0$. Otherwise the
    third valuation is $t+s$ for some $s>0$, and the same formulas,
    using $\mathfrak m\nmid2ABC$, give
    $v_{\mathfrak m}(j_E)=-2ps$. Since $\mathfrak m\nmid p$,
    Lemma~\ref{criteria for potentially multiplicative reduction}
    yields in either case
    \[
        p\nmid\#\bar\rho_{E,p}(I_{\mathfrak m}).
    \]

    By Shafarevich's theorem \cite[Theorem~IX.6.1]{S09}, elliptic curves
    over $K$ with good reduction outside $S_K'\cup\mathcal H$
    form finitely many isomorphism classes over $\bar K$.
    Hence the possible values of $j_{E'}$ form a fixed finite set.
    Since $\mathcal H$ is also finite, we may increase $V$
    so that, whenever $v_{\mathfrak m}(j_{E'})<0$, one has
    $0<-v_{\mathfrak m}(j_{E'})<p$.
    Lemma~\ref{criteria for potentially multiplicative reduction}(2)(a)
    would then give
    $p\mid\#\bar\rho_{E',p}(I_{\mathfrak m})$,
    contrary to $\bar\rho_{E',p}\simeq\bar\rho_{E,p}$.
    Thus $v_{\mathfrak m}(j_{E'})\geq0$. Since $E'$ already has
    good reduction outside $S_K'\cup\{\mathfrak m\}$ , it has potentially good
    reduction away from $S_K'$.
\end{proof}

\subsection{Proof of the main results}
We are now ready to prove Theorem~\ref{main result1 for Ax^p+By^p=Cz^p}.
\begin{proof}[Proof of Theorem~\ref{main result1 for Ax^p+By^p=Cz^p}]
    Fix the prime $\mfP$ in the statement, and suppose that
    $(a,b,c)\in K^3$ is a non-trivial solution with sufficiently
    large prime exponent $p$. Normalize it as at the beginning of
    \S\ref{steps to prove main results}. Since the normalized
    coordinate ideal is $\mathfrak m\notin S_K'$, we have
    \[
        \min\{v_{\mfP}(a),v_{\mfP}(b),v_{\mfP}(c)\}
        =v_{\mfP}(\mathfrak m)=0.
    \]
    Common scaling adds the same integer to all three valuations,
    so the condition $(a,b,c)\in W_{K,\mfP}$ is unchanged.
    In part~(1), the coefficient condition
    $
    v_{\mfP}(A)=v_{\mfP}(B)=m,\
    v_{\mfP}(C)=m+k,
    $ and
    Lemma~\ref{local reduction without divisibility} give
    \eqref{local conditions for modular argument}, i.e.,  $v_{\mfP}(j_E)<0$ and $ p\nmid v_{\mfP}(j_E)$.
    In part~(2), we have $(a,b,c)\in W_{K,\mfP}$ and
    $\alpha_{\mfP}\beta_{\mfP}\gamma_{\mfP}\neq0$, so
    Lemma~\ref{reduction on T and S} gives
    \eqref{local conditions for modular argument}.
    In either case, Theorem~\ref{auxilary result x^2=By^p+Cz^p over W_K}
    therefore produces an elliptic curve $E'/K$ with full
    rational $2$-torsion, potentially good reduction away from $S_K'$, and
    \begin{equation}
        \label{negative valuation E prime}
        v_{\mfP}(j_{E'})<0.
    \end{equation}
    We derive the same $S_K'$-unit contradiction in both cases.

    Write
    $E':Y^2=(X-e_1)(X-e_2)(X-e_3)$, with distinct
    $e_1,e_2,e_3\in K$, and put
    $\lambda=(e_3-e_1)/(e_2-e_1)$.
    Then $\lambda\in K\setminus\{0,1\}$ and
    \begin{equation}
        \label{j'-invariant of Legendre form}
        j_{E'}=
        2^8\frac{(\lambda^2-\lambda+1)^3}
                   {\lambda^2(1-\lambda)^2}.
    \end{equation}
    Since $j_{E'}$ is $S_K'$-integral and
    $2\in\mcO_{S_K'}^\times$, the monic polynomial
    \[
    \begin{aligned}
        X^6-3X^5
        &+\left(6-\frac{j_{E'}}{2^8}\right)X^4
         +\left(\frac{2j_{E'}}{2^8}-7\right)X^3\\
        &+\left(6-\frac{j_{E'}}{2^8}\right)X^2-3X+1
    \end{aligned}
    \]
    has coefficients in $\mcO_{S_K'}$.
    It vanishes at the Legendre parameters
    $\lambda$, $\lambda^{-1}$, $1-\lambda$, and
    $(1-\lambda)^{-1}$.
    By taking $\mu=1-\lambda$, we have
    $
        \lambda,\mu\in\mcO_{S_K'}^\times,\
        \lambda+\mu=1.
    $
    Equation~\eqref{j'-invariant of Legendre form} becomes
    \begin{equation}
        \label{j' in terms of lambda and mu}
        j_{E'}=
        2^8\frac{(1-\lambda\mu)^3}{(\lambda\mu)^2}.
    \end{equation}

    Put $e=v_{\mfP}(2)$ and
    \[
        s=\max\{|v_{\mfP}(\lambda)|,|v_{\mfP}(\mu)|\}.
    \]
    The common hypothesis
    \eqref{assumption for main result x^p+y^p=2^rz^p}
    gives $s\leq4e$.
    If $s=0$, then $v_{\mfP}(\lambda\mu)=0$ and
    $v_{\mfP}(1-\lambda\mu)\geq0$, so
    \eqref{j' in terms of lambda and mu} gives
    $v_{\mfP}(j_{E'})\geq8e>0$.
    If $s>0$, then the relation $\lambda+\mu=1$ implies
    \[
        (v_{\mfP}(\lambda),v_{\mfP}(\mu))
           \in\{(-s,-s),(s,0),(0,s)\}.
    \]
    In the first case the valuations of $\lambda\mu$ and
    $1-\lambda\mu$ are both $-2s$; in the other two cases
    they are $s$ and $0$, respectively.
    Hence, in every case with $s>0$,
    \[
        v_{\mfP}(j_{E'})=8e-2s\geq0.
    \]
    This contradicts \eqref{negative valuation E prime}.
Hence the proof follows.
\end{proof}
We are now ready to prove Theorem~\ref{totally real unrestricted criterion}.
\begin{proof}[Proof of Theorem~\ref{totally real unrestricted criterion}]
    Let $(a,b,c)$ be a non-trivial solution in $K^3$ in part~(1),
    or in $W_{K,\mfP}$ in part~(2), with sufficiently large prime
    exponent $p$. Normalize it as at the beginning of
    \S\ref{steps to prove main results}. Since
    $a\mcO_K+b\mcO_K+c\mcO_K=\mathfrak m\notin S_K'$, we have
    $\min\{v_{\mfP}(a),v_{\mfP}(b),v_{\mfP}(c)\}=0$.
    For part~(1) (resp., part~(2)), Lemma~\ref{local reduction without divisibility} (resp., Lemma~\ref{reduction on T and S} ) gives
    \[
        v_{\mfP}(j_E)<0,\qquad
        p\nmid v_{\mfP}(j_E),\qquad
        p\mid\#\bar\rho_{E,p}(I_{\mfP}).
    \]

    We first establish modularity. In part~(1), a relation
    $\epsilon_1A+\epsilon_2B+\epsilon_3C=0$, with
    $\epsilon_i\in\{\pm1\}$, would give an $S_K'$-unit solution
    \[
        \lambda=-\frac{\epsilon_1A}{\epsilon_3C},\qquad
        \mu=-\frac{\epsilon_2B}{\epsilon_3C},\qquad
        v_{\mfP}(\lambda)=v_{\mfP}(\mu)=-k,
    \]
    contradicting $k>4v_{\mfP}(2)$. Thus $A\pm B\pm C\neq0$, and
    \cite[Theorem~3.2]{S26}, applied to $A,B,-C$, gives modularity
    for sufficiently large $p$.

    In part~(2), Lemma~\ref{reduction on T and S} gives
    \[
        v_{\mfP}(j_E)\leq D_{\mfP}-2p,\qquad
        D_{\mfP}=\max\{\alpha_{\mfP},\beta_{\mfP},\gamma_{\mfP}\}.
    \]
    By \cite[Theorem~5]{FS15}, the nonzero non-modular
    $j$-invariants over $K$ form a finite set $\{j_1,\ldots,j_r\}$.
    If $r>0$, then taking
    $
        p>\frac12\left(D_{\mfP}
            -\min_{1\leq i\leq r}v_{\mfP}(j_i)\right)
    $
    gives $v_{\mfP}(j_E)<\min_i v_{\mfP}(j_i)$, excluding all
    exceptional values. Since $v_{\mfP}(j_E)<0$ also excludes
    $j_E=0$, the curve $E$ is modular; this is immediate if $r=0$.
    This bound is uniform but generally ineffective, since the
    exceptional $j$-invariants are not explicitly known.

    For both parts, take $p>24$ unramified in $K$, with no prime above $p$ in
    $S_K'\cup\mathcal H$. Lemmas~\ref{Serre's conductor}
    and~\ref{surjective of mod p GR} verify the hypotheses of
    Remark~\ref{remark totally real level lowering}. Thus there is
    a Hilbert newform $f$ of parallel weight $2$ at the Serre
    conductor $\mfn$ such that
    $\bar\rho_{E,p}\simeq\bar\rho_{f,\lambda}$, with $\lambda\mid p$.
    The possible levels form a fixed finite set, so the rationality
    argument in the proof of
    Theorem~\ref{auxilary result x^2=By^p+Cz^p over W_K}
    gives $\Q_f=\Q$ after increasing the bound on $p$.

    The local conditions at $\mfP$ allow us to apply
    \cite[Theorem~3.6]{S26}, after excluding
    $p\mid(\operatorname{Norm}_{K/\Q}(\mfP)^2-1)$.
    This supplies an elliptic curve $E_f/K$ of conductor $\mfn$
    with $\bar\rho_{E_f,p}\simeq\bar\rho_{E,p}$, making the
    Eichler--Shimura step in Remark~\ref{remark totally real ES}
    unconditional. The full-$2$-torsion and inertia arguments in
    the proof of Theorem~\ref{auxilary result x^2=By^p+Cz^p over W_K}
    now apply: after increasing the bound on $p$ and replacing
    $E_f$ by a $2$-isogenous curve $E'$ over $K$ if necessary,
    we get $E'/K$ with
    \[
        E'[2]\subseteq E'(K),\qquad
        \bar\rho_{E',p}\simeq\bar\rho_{E,p},\qquad
        j_{E'}\in\mcO_{S_K'},\qquad v_{\mfP}(j_{E'})<0.
    \]
    These arguments use only the residual isomorphism, the local
    Frey invariants and Shafarevich finiteness, with no conjectural
    assumption. The Legendre-parameter argument in the proof of
    Theorem~\ref{main result1 for Ax^p+By^p=Cz^p}, together with
    the assumed $S_K'$-unit bound, gives $v_{\mfP}(j_{E'})\geq0$,
    a contradiction.
\end{proof}

\begin{proof}[Proof of Proposition~\ref{S unit crit 2d}]
	Put $T:=\{\mfq\in P:\mfq\mid2d\}$. Since
	$A,B,C\in\{u2^rd^s:u\in\mcO_K^\times,\ r,s\in\Z_{\geq0}\}$,
	we have $S_K'\subseteq T$. By assumption, $2$ is inert in $K$;
    write $2\mcO_K=\mfP$, so $\mfP^5=32\mcO_K$.
    If $d\equiv u^2\pmod{\mfP^5}$ for some $u\in\mcO_K$,
    then taking norms implies
    \[
        d^{[K:\Q]}\equiv N_{K/\Q}(u)^2\pmod{32}.
    \]
    Since $N_{K/\Q}(u)$ is odd, its square is congruent to
    $1,9,17$, or $25$ modulo $32$, contradicting the hypothesis.
    Thus $d$ is not a square modulo $\mfP^5$.
    By \cite[Proposition~3.5]{S26b}, we get that
    every solution $(\lambda,\mu)$ of the $T$-unit equation satisfies
	\[
	\max\{|v_\mfP(\lambda)|,|v_\mfP(\mu)|\}\leq4=4v_\mfP(2).
	\]
	Every $S_K'$-unit is a $T$-unit, so the same bound holds for every
    solution of the $S_K'$-unit equation. Both conclusions now follow
    from Theorem~\ref{main result1 for Ax^p+By^p=Cz^p}.
\end{proof}
\begin{proof}[Proof of Corollary~\ref{irrelevant cor}]
		Recall that the irrelevant solutions to the $S_K^\prime$-unit
		equation~\eqref{S_K-unit solution} are
		$
		(2,-1),\  (-1,2),\
		\left(\frac{1}{2},\frac{1}{2} \right)$. In each of these cases, for every
		$\mfP\in S_K$ we have
		\[
		\max\left\{
		|v_{\mfP}(\lambda)|,
		|v_{\mfP}(\mu)|
		\right\}
		= v_{\mfP}(2) \leq 4v_{\mfP}(2).
		\]
		Thus every solution of the $S_K^\prime$-unit equation satisfies the
		hypothesis~\eqref{assumption for main result x^p+y^p=2^rz^p} of
		Theorem~\ref{main result1 for Ax^p+By^p=Cz^p}. Both
        conclusions therefore follow immediately from that theorem.
        In the totally real case, the same bound allows us to apply
        Theorem~\ref{totally real unrestricted criterion}, giving
        both stated conclusions.
\end{proof}

\section{Quadratic field applications and density}
\subsection{Imaginary quadratic applications and density}
\label{section for loc criteria}
We first determine the solutions of the $S_K$-unit equation and then apply both
criteria.

\begin{proof}[Proof of Theorem~\ref{imaginary quadratic generalized Fermat}]
    The conditions on $d$ imply that $2$ is either inert or totally
    ramified in $K$, so $S_K=\{\mfP\}$.
    For coefficients that are signed powers of $2$, one has
    $S_K'=S_K=\{\mfP\}$.
    Moreover, $\mcO_K^\times=\{\pm1\}$ because $d\geq5$.

    If $-d\equiv5\pmod8$, then $2$ is inert and $\mfP=(2)$.
    Every $S_K'$-unit $u$ satisfies
    $(u)=\mfP^n=(2^n)$ for some $n\in\Z$, so $u=\pm2^n$.
    If $-d\equiv2,3\pmod4$, then
    $\mcO_K=\Z[\sqrt{-d}]$ and $\mfP^2=(2)$.
    The ideal $\mfP$ is not principal: a generator
    $a+b\sqrt{-d}$ would have norm $2$, giving
    $a^2+db^2=2$ with $a,b\in\Z$, which is impossible.
    Thus the class of $\mfP$ has order exactly $2$.
    If $(u)=\mfP^n$ is principal, then $n=2t$ for some
    $t\in\Z$, so $(u)=(2^t)$ and $u=\pm2^t$.

    Consequently, in either case,
    $
        \mcO_{S_K'}^\times=\{\pm2^n:n\in\Z\}.
    $
    Solving $\lambda+\mu=1$ in this group gives
    \[
        (\lambda,\mu)=(2,-1),\qquad(-1,2),\qquad
                         \left(\frac12,\frac12\right).
    \]
    In particular, every solution satisfies
    \[
        \max\{|v_{\mfP}(\lambda)|,|v_{\mfP}(\mu)|\}
            =v_{\mfP}(2)\leq4v_{\mfP}(2).
    \]

    For part~(1),
    take $A=B=1$ and $C=2^r$.
    With $e=v_{\mfP}(2)$, the coefficient valuations are
    $(0,0,re)$ and $re>4e$ for each fixed $r\geq5$.
    Theorem~\ref{main result1 for Ax^p+By^p=Cz^p}(1) therefore
    excludes all non-trivial solutions in $K^3$ for
    sufficiently large prime exponents.
    For part~(2), the assumed nonvanishing of
    $\alpha_{\mfP}\beta_{\mfP}\gamma_{\mfP}$ allows us to apply
    Theorem~\ref{main result1 for Ax^p+By^p=Cz^p}(2), hence the proof follows.
\end{proof}
To prove Proposition~\ref{thm for density}, we first recall the
notion of absolute density.

\begin{dfn}
	Let $S\subseteq\N$, and for $x>0$ define
	$
	S(x):=\{d\in S:d\leq x\}.$
	The \emph{absolute density} of $S$ is defined by
	\[
	\delta(S)
	:=
	\lim_{x\to\infty}\frac{\#S(x)}{x},
	\]
	provided that the limit exists.
\end{dfn}

We shall use the following result of Freitas and Siksek.

\begin{thm}[{\cite[Theorem~10]{FS15}}]
	\label{FS thm for density}
	For $r\in\Z$ and $m\in\N$, let
	$
	\N_{r,m}^{\mathrm{sf}}
	:=
	\{d\in\N^{\mathrm{sf}}:d\equiv r\pmod m\}.$
	If $s:=\gcd(r,m)$ is squarefree, then
	\[
	\#\N_{r,m}^{\mathrm{sf}}(x)
	\sim
	\frac{\varphi(m)}
	{s\varphi(m/s)m
		\displaystyle\prod_{q\mid m}
		\left(1-\frac{1}{q^2}\right)}
	\cdot\frac{6}{\pi^2}x,
	\]
	where $\varphi$ denotes Euler's totient function and the product
	runs over all prime divisors $q$ of $m$.
\end{thm}

We are now ready to prove Proposition~\ref{thm for density}.

\begin{proof}[Proof of Proposition~\ref{thm for density}]
	Recall that
	$
	U
	:=
	\left\{
	d\in\N^{\mathrm{sf}}:
	d\geq5\text{ and }-d\not\equiv1\pmod8
	\right\}.
	$
	Let
	$
	V
	:=
	\left\{
	d\in\N^{\mathrm{sf}}:
	-d\equiv1\pmod8
	\right\}.
	$
	Since $-d\equiv1\pmod8$ if and only if $d\equiv7\pmod8$, we have
	$
	V=\N_{7,8}^{\mathrm{sf}}.
	$
	Moreover,
	$
	\N^{\mathrm{sf}}
	=
	\{2,3\}\cup U\cup V.$
	The finite set $\{2,3\}$ does not affect density. Therefore,
	$
	\delta_{\mathrm{rel}}(U)
	=
	1-\delta_{\mathrm{rel}}(V).
	$

	We first compute the absolute density of $\N^{\mathrm{sf}}$.
	Since
	$
	\N^{\mathrm{sf}}=\N_{0,1}^{\mathrm{sf}},
	$
	Theorem~\ref{FS thm for density} gives
	$
	\#\N^{\mathrm{sf}}(x)
	\sim
	\frac{6}{\pi^2}x.
	$
	Hence
	$
	\delta(\N^{\mathrm{sf}})
	=
	\frac{6}{\pi^2}.
	$
	Consequently, for any $S\subseteq\N^{\mathrm{sf}}$ for which the
	relevant density exists,
	$
	\delta_{\mathrm{rel}}(S)
	=
	\frac{\delta(S)}
	{\delta(\N^{\mathrm{sf}})}
	=
	\frac{\pi^2}{6}\delta(S).
	$

	It remains to compute the absolute density of $V$. Since
	$
	V=\N_{7,8}^{\mathrm{sf}},
	$
	we apply Theorem~\ref{FS thm for density} with
	$
	r=7,\ m=8,\ s=\gcd(7,8)=1.
	$
	We have
	$
	\varphi(8)=4$ and the only prime divisor of $8$ is $2$. Therefore,
	\begin{align*}
		\#\N_{7,8}^{\mathrm{sf}}(x)
		&\sim
		\frac{4}
		{1\cdot4\cdot8
			\left(1-\frac{1}{2^2}\right)}
		\cdot\frac{6}{\pi^2}x\\
		&=
		\frac{4}
		{32\cdot\frac34}
		\cdot\frac{6}{\pi^2}x\\
		&=
		\frac{1}{\pi^2}x.
	\end{align*}
	Thus
	$
	\delta(V)
	=
	\delta(\N_{7,8}^{\mathrm{sf}})
	=
	\frac{1}{\pi^2}.
	$
	It follows that
	$
	\delta_{\mathrm{rel}}(V)
	=
	\frac{\delta(V)}
	{\delta(\N^{\mathrm{sf}})}
	=
	\frac{1/\pi^2}{6/\pi^2}
	=
	\frac16.
	$
	Therefore,
	$
		\delta_{\mathrm{rel}}(U)
		=
		1-\frac16
		=
		\frac56.
	$
	This completes the proof.
\end{proof}

\subsection{Real quadratic applications and density}
\label{section real quadratic proofs}
\begin{proof}[Proof of Theorem~\ref{real quadratic generalized Fermat}]
    Let $\mathcal C$ be the set of squarefree integers $d\geq2$
    for which the $S_d$-unit equation has only the three solutions
    $
        (\lambda,\mu)=(2,-1),\ (-1,2),\
        \left(\frac12,\frac12\right).
    $
    By \cite[Theorem~4]{FS15},
    $\delta_{\mathrm{rel}}(\mathcal C)=1$.
    For each $d\in\mathcal C$, these solutions satisfy
    \[
        \max\{|v_{\mfP}(\lambda)|,|v_{\mfP}(\mu)|\}
            =v_{\mfP}(2)\leq4v_{\mfP}(2)
        \qquad\text{for every }\mfP\in S_d.
    \]
    Thus $\mathcal C\subseteq\mathcal G$, and consequently
    $\delta_{\mathrm{rel}}(\mathcal G)=1$.

    Now fix $d\in\mathcal G$, put $K=\Q(\sqrt d)$, and choose
    $\mfP\in S_K$ satisfying the coefficient condition in
    the statement. Since $A,B,C$ have no prime divisors away
    from $2$, we have
    \[
        S_K'=\{\mfq:\mfq\mid2ABC\}
             =S_K=S_d.
    \]
    Since $d\in\mathcal G$, every $S_K'$-unit solution $(\lambda,\mu)$ satisfies
    \[
        \max\{|v_{\mfP}(\lambda)|,|v_{\mfP}(\mu)|\}
        \leq4v_{\mfP}(2).
    \]
    As $K$ is totally real, the coefficient hypothesis allows us
    to apply Theorem~\ref{totally real unrestricted criterion}(1).
    Hence $Ax^p+By^p=Cz^p$ has no asymptotic solution in $K^3$.

    Taking $A=B=1$ and $C=2^r$, with $r\geq5$, gives $m=0$
    and $k=rv_{\mfP}(2)>4v_{\mfP}(2)$, hence the final assertion follows.
\end{proof}

\end{document}